\documentclass{gtpart}
\usepackage{pinlabel}
\usepackage{graphicx}
\title[Root systems, symmetries and linear representations of Artin groups]{Root systems, symmetries and linear representations\\ of Artin groups}
\author[O Geneste]{Olivier Geneste}
\givenname{Olivier}
\surname{Geneste}
\address{IMB, UMR 5584, CNRS, Université Bourgogne Franche-Comté, 21000 Dijon, France}
\email{o.geneste@gmail.com}
\urladdr{}
\author[J-Y Hée]{Jean-Yves Hée}
\givenname{Jean-Yves}
\surname{Hée}
\address{LAMFA, UMR 7352, CNRS, Université de Picardie Jules Verne, 80039 Amiens, France}
\email{jean-yves.hee@u-picardie.fr}
\urladdr{}
\author[L Paris]{Luis Paris}
\givenname{Luis}
\surname{Paris}
\address{IMB, UMR 5584, CNRS, Université Bourgogne Franche-Comté, 21000 Dijon, France}
\email{lparis@u-bourgogne.fr}
\urladdr{}
\subject{primary}{msc2000}{20F36}
\subject{primary}{msc2000}{20F55}

\volumenumber{}
\issuenumber{}
\publicationyear{}
\papernumber{}
\startpage{}
\endpage{}
\doi{}
\MR{}
\Zbl{}
\received{}
\revised{}
\accepted{}
\published{}
\publishedonline{}
\proposed{}
\seconded{}
\corresponding{}
\editor{}
\version{}
\newtheorem{thm}{Theorem}[section]
\newtheorem{lem}[thm]{Lemma}
\newtheorem{prop}[thm]{Proposition}
\newtheorem{corl}[thm]{Corollary}

\theoremstyle{definition}

\newtheorem*{expl}{Example}

\numberwithin{equation}{section}

\makeatletter
\renewcommand{\thefigure}{\ifnum \c@section>\z@ \thesection.\fi
 \@arabic\c@figure}
\@addtoreset{figure}{section}
\makeatother

\begin{document}

\def\BB{\mathcal B} \def\GL{{\rm GL}} \def\K{\mathbb K}
\def\Q{\mathbb Q} \def\N{\mathbb N} \def\B{\mathbb B}
\def\RR{\mathcal R} \def\R{\mathbb R} \def\OO{\mathcal O}
\def\SS{\mathcal S} \def\Sym{{\rm Sym}} \def\Stab{{\rm Stab}}
\def\Supp{{\rm Supp}} \def\id{{\rm id}} \def\fin{{\rm fin}}
\def\SSS{\mathfrak S}


\begin{abstract}
Let $\Gamma$ be a Coxeter graph, let $W$ be its associated Coxeter group, and let $G$ be a group of symmetries of $\Gamma$.
Recall that, by a theorem of Hée and M\"uhlherr, $W^G$ is a Coxeter group associated to some Coxeter graph $\hat \Gamma$.
We denote by $\Phi^+$ the set of positive roots of $\Gamma$ and by $\hat \Phi^+$ the set of positive roots of $\hat \Gamma$.
Let $E$ be a vector space over a field $\K$ having a basis in one-to-one correspondence with $\Phi^+$.
The action of $G$ on $\Gamma$ induces an action of $G$ on $\Phi^+$, and therefore on $E$.
We show that $E^G$ contains a linearly independent family of vectors naturally in one-to-one correspondence with $\hat \Phi^+$ and we determine exactly when this family is a basis of $E^G$.
This question is motivated by the construction of Krammer's style linear representations for non simply laced Artin groups. 
\end{abstract}

\maketitle


\section{Introduction}

\subsection{Motivation}\label{subsec1_1}

Bigelow \cite{Bige1} and Krammer \cite{Kramm1} proved that the braid groups are linear answering a historical question in the subject. 
More precisely, they proved that some linear representation $\psi : \BB_n \to \GL (E)$ of the braid group $\BB_n$ previously introduced by Lawrence \cite{Lawre1} is faithful.
A useful information for us is that $E$ is a vector space over the field $\K = \Q(q,z)$ of rational functions in two variables $q,z$ over $\Q$, and has a natural basis of the form $\{e_{i,j} \mid 1 \le i < j \le n\}$.

Let $\Gamma$ be a Coxeter graph, let $W_\Gamma$ be its associated Coxeter group, let $A_\Gamma$ be its associated Artin group, and let $A_\Gamma^+$ be its associated Artin monoid. 
The Coxeter graph $\Gamma$ is called of \emph{spherical type} if $W_\Gamma$ is finite, it is called \emph{simply laced} if none of its edges is labelled, and it is called \emph{triangle free} if there are no three vertices in $\Gamma$ two by two connected by edges. 
Shortly after the release of the papers by Bigelow \cite{Bige1} and Krammer \cite{Kramm1}, Digne \cite{Digne1} and independently Cohen--Wales \cite{CohWal1} extended Krammer's \cite{Kramm1} constructions and proofs to the Artin groups associated with simply laced Coxeter graphs of spherical type, and, afterwards, Paris \cite{Paris1} extended them to all the Artin groups associated to simply laced triangle free Coxeter graphs (see also Hée \cite{Hee1} for a simplified proof of the faithfulness of the representation).
More precisely, for a finite simply laced triangle free Coxeter graph $\Gamma$, they constructed a linear representation $\psi : A_\Gamma \to \GL(E)$, they showed that this representation is always faithful on the Artin monoid $A^+_\Gamma$, and they showed that it is faithful on the whole group $A_\Gamma$ if $\Gamma$ is of spherical type.
What is important to know here is that $E$ is still a vector space over $\K = \Q(q,z)$ and that $E$ has a natural basis $\BB = \{e_\beta \mid \beta \in \Phi^+ \}$ in one-to-one correspondence with the set $\Phi^+$ of positive roots of $\Gamma$.

The question that motivated the beginning of the present study is to find a way to extend the construction of this linear representation to other Artin groups, or, at least, to some Artin groups whose Coxeter graphs are not simply laced and triangle free. 
A first approach would be to extend Paris' \cite{Paris1} construction to other Coxeter graphs that are not simply laced and triangle free. 
Unfortunately, explicit calculations on simple examples convinced us that this approach does not work.

However, an idea for constructing such linear representations for some Artin groups associated to non simply laced Coxeter graphs can be found in Digne \cite{Digne1}.
In that paper Digne takes a Coxeter graph $\Gamma$ of type $A_{2n+1}$, $D_n$, or $E_6$ and consider some specific symmetry $g$ of $\Gamma$. 
By H\'ee \cite{Hee2} and M\"uhlherr \cite{Muhlh1} the subgroup $W_\Gamma^g$ of fixed elements by $g$ is itself a Coxeter group associated with a precise Coxeter graph $\hat \Gamma$.
By Michel \cite{Miche1}, Crisp \cite{Crisp1, Crisp2} and Dehornoy--Paris \cite{DehPar1}, the subgroup $A_\Gamma^g$ of $A_\Gamma$ of fixed elements by $g$ is an Artin group associated with $\hat \Gamma$.
On the other hand the symmetry $g$ acts on the basis $\BB$ of $E$ and the linear representation $\psi : A_\Gamma \to \GL(E)$ is equivariant in the sense that $\psi(g(a)) = g\, \psi(a) \, g^{-1}$ for all $a \in A_\Gamma$.
It follows that $\psi$ induces a linear representation $\psi^g : A_{\hat\Gamma} \to \GL(E^g)$, where $E^g$ denotes the subspace of fixed vectors of $E$ under the action of $g$. 
Then Digne \cite{Digne1} proves that $\psi^g$ is faithful and that $E^g$ has a ``natural'' basis in one-to-one correspondence with the set $\hat \Phi^+$ of positive roots of $\hat \Gamma$.
This defines a linear representation for the Artin groups associated with the Coxeter graphs $B_n$ ($n \ge 2$), $G_2$ and $F_4$.

Let $\Gamma$ be a finite simply laced triangle free Coxeter graph and let $G$ be a non-trivial group of symmetries of $\Gamma$.
Then $G$ acts on the groups $W_\Gamma$ and $A_\Gamma$ and on the monoid $A_\Gamma^+$.
We know by Hée \cite{Hee2} and M\"uhlherr \cite{Muhlh1} (see also Crisp \cite{Crisp1, Crisp2}, Geneste--Paris \cite{GenPar1} and Theorem \ref{thm2_5}) that $W_\Gamma^G$ is the Coxeter group associated with some precise Coxeter graph $\hat \Gamma$. 
Moreover, by Crisp \cite{Crisp1, Crisp2}, the monoid $A_\Gamma^{+G}$ is an Artin monoid associated with $\hat \Gamma$ and in many cases the group $A_\Gamma^G$ is an Artin group associated with $\hat \Gamma$.  
On the other hand, $G$ acts on the basis $\BB$ of $E$, and the linear representation $\psi: A_\Gamma \to \GL(E)$ is equivariant in the sense that $\psi(g(a)) = g\, \psi(a) \, g^{-1}$ for all $a \in A_\Gamma$ and all $g \in G$.
Thus, $\psi$ induces a linear representation $\psi^G : A_{\hat \Gamma} \to \GL(E^G)$, where $E^G = \{ x \in E \mid g(x) = x \text{ for all } g \in G \}$.
We also know by Castella \cite{Caste2, Caste3} that the induced representation $\psi^G : A_{\hat\Gamma} \to \GL(E^G)$ is faithful on the monoid $A_{\hat \Gamma}^+$.
So, it remains to determine when $E^G$ has a ``natural'' basis in one-to-one correspondence with the set $\hat \Phi^+$ of positive roots of $\hat \Gamma$. 
The purpose of this paper is to answer this question.

\subsection{Statements}

The simply laced triangle free Artin groups and the linear representations $\psi : A_\Gamma \to \GL(E)$ form the framework of our motivation, but they are not needed for the rest of the paper. 
We will also work with any Coxeter graph, which may have labels and infinitely many vertices. 
So, let $\Gamma$ be a Coxeter graph associated with a Coxeter matrix $M = (m_{s,t})_{s,t \in S}$, let $\K$ be a field, and let $E$ be a vector space over $\K$ having a basis $\BB=\{ e_\beta \mid \beta \in \Phi^+ \}$ in one-to-one correspondence with the set $\Phi^+$ of positive roots of $\Gamma$.

A \emph{symmetry} of $\Gamma$ is defined to be a permutation $g$ of $S$ satisfying $m_{g(s),g(t)} = m_{s,t}$ for all $s,t \in S$.
The group of symmetries of $\Gamma$ will be denoted by $\Sym (\Gamma)$.
Let $G$ be a subgroup of $\Sym(\Gamma)$.
Again, we know by Hée \cite{Hee2} and M\"uhlherr \cite{Muhlh1} that $W_\Gamma^G$ is the Coxeter group associated with some Coxeter graph $\hat \Gamma$. 
On the other hand, $G$ acts on the set $\Phi^+$ of positive roots of $\Gamma$ and therefore on $E$.
Let $\hat \Phi^+$ be the set of positive roots of $\hat \Gamma$.
In this paper we show that $E^G$ contains a ``natural'' linearly independent set $\hat \BB = \{ \hat e_{\hat \beta} \mid \hat \beta \in \hat \Phi^+\}$ in one-to-one correspondence with the set $\hat \Phi^+$ and we determine when $\hat \BB$ is a basis of $E^G$.

From now on we will say that the pair $(\Gamma, G)$ has the \emph{$\hat \Phi^+$-basis property} if the above mentioned subset $\hat \BB$ is a basis of $E^G$.

We proceed in three steps to determine the pairs $(\Gamma, G)$ that have the $\hat\Phi^+$-basis property. 
In a first step (see Subsection \ref{subsec3_1}) we show that it suffices to consider the case where all the orbits of $S$ under the action of $G$ are finite. 
Let $S_\fin$ denote the union of the finite orbits of $S$ under the action of $G$, and let $\Gamma_\fin$ denote the full subgraph of $\Gamma$ spanned by $S_\fin$.
Each symmetry $g \in G$ stabilizes $S_\fin$, hence induces a symmetry of $\Gamma_\fin$.
We denote by $G_\fin$ the subgroup of $\Sym (\Gamma_\fin)$ of all these symmetries.
In Subsection \ref{subsec3_1} we will prove the following.

\begin{thm}\label{thm2_9}
The pair $(\Gamma, G)$ has the $\hat \Phi^+$-basis property if and only if the pair $(\Gamma_\fin, G_\fin)$ has the $\hat \Phi^+$-basis property.
\end{thm}

By Theorem \ref{thm2_9} we can assume that all the orbits of $S$ under the action of $G$ are finite.
In a second step (see Subsection \ref{subsec3_2}) we show that it suffices to consider the case where $\Gamma$ is connected.
Let $\Gamma_i$, $i \in I$, be the connected components of $\Gamma$.
For each $i \in I$ we denote by $\Stab_G (\Gamma_i)$ the stabilizer of $\Gamma_i$ in $G$.
Each symmetry $g \in \Stab_G (\Gamma_i)$ induces a symmetry of $\Gamma_i$.
We denote by $G_i$ the subgroup of $\Sym (\Gamma_i)$ of all these symmetries.
In Subsection \ref{subsec3_2} we will prove the following.

\begin{thm}\label{thm2_10}
Suppose that all the orbits of $S$ under the action of $G$ are finite.
Then the pair $(\Gamma, G)$ has the $\hat \Phi^+$-basis property if and only if for each $i \in I$ the pair $(\Gamma_i, G_i)$ has the $\hat \Phi^+$-basis property.
\end{thm}

By Theorem \ref{thm2_10} we can also assume that $\Gamma$ is connected.
In a third step (see Subsection \ref{subsec3_3} and Subsection \ref{subsec3_4}) we determine all the pairs $(\Gamma, G)$ that have the $\hat \Phi^+$-basis property with $\Gamma$ connected and all the orbits of $S$ under the action of $G$ being finite. 

One can associate with $\Gamma$ a real vector space $V = \oplus_{s \in S} \R \alpha_s$ whose basis is in one-to-one correspondence with $S$ and a \emph{canonical bilinear form} $\langle .,. \rangle : V \times V \to \R$.
These objects will be defined in Subsection \ref{subsec2_1}.
We say that $\Gamma$ is of \emph{spherical type} if $S$ is finite and $\langle .,. \rangle$ is positive definite, and we say that $\Gamma$ is of \emph{affine type} if $S$ is finite and $\langle .,. \rangle$ is positive but not positive definite. 
A classification of the connected spherical and affine type Coxeter graphs can be found in Bourbaki \cite{Bourb1}.
In this paper we use the notations $A_m$ ($m \ge 1$), ..., $I_2(p)$ ($p=5$ or $p \ge 7$) of Bourbaki \cite[Chap. VI, Parag. 4, No 1, Th\'eor\`eme 1]{Bourb1} for the connected  Coxeter graphs of spherical type, and the notations $\tilde A_1$, $\tilde A_m$ ($m\ge 2$), ..., $\tilde G_2$ of Bourbaki \cite[Chap. VI, Parag. 4, No 2, Th\'eor\`eme 4]{Bourb1} for the connected  Coxeter graphs of affine type.
Moreover, we use the same numbering of the vertices of these Coxeter graphs as the one in Bourbaki \cite[Planches]{Bourb1} with the convention that the unnumbered vertex in Bourbaki \cite[Planches]{Bourb1} is here labelled with $0$.

To the Coxeter graphs of spherical type and affine type we must add the two infinite Coxeter graphs ${}_\infty A_\infty$ and $D_\infty$ drawn in Figure \ref{fig2_1}.
The Coxeter graph $A_\infty$ of the figure does not appear in the statement of Theorem \ref{thm2_12} but it will appear in its proof.
These Coxeter graphs are part of the family of so-called \emph{locally spherical Coxeter graphs} studied by Hée \cite[Texte 10]{Hee3}.

\begin{figure}[ht!]
\begin{center}
\begin{tabular}{cc}
\parbox[c]{3.8cm}{\includegraphics[width=3.6cm]{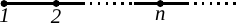}}&
\parbox[c]{4.6cm}{\includegraphics[width=4.4cm]{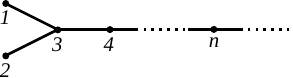}}\\
$A_\infty$&
$D_\infty$
\end{tabular}

\bigskip
\begin{tabular}{c}
\parbox[c]{7.4cm}{\includegraphics[width=7.2cm]{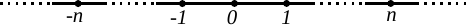}}\\
\noalign{\smallskip}
${}_\infty A_\infty$
\end{tabular}

\caption{Locally spherical Coxeter graphs}\label{fig2_1}
\end{center}
\end{figure}

Now, the conclusion of the third step which is in some sense the main result of the paper is the following theorem, proved in Subsection \ref{subsec3_4}.

\begin{thm}\label{thm2_12}
Suppose that $\Gamma$ is connected, $G$ is non-trivial and all the orbits of $S$ under the action of $G$ are finite.
Then $(\Gamma,G)$ has the $\hat \Phi^+$-basis property if and only if $(\Gamma, G)$ is up to isomorphism one of the following pairs (see Figure \ref{fig2_2}, Figure \ref{fig2_3} and Figure \ref{fig2_4}).
\begin{itemize}
\item[(i)]
$\Gamma = A_{2 m +1}$ ($m \ge 1$) and $G = \langle g \rangle$ where
\[
g(s_i) = s_{2 m +2 -i}\ (1 \le i \le 2 m + 1)\,.
\]
\item[(ii)]
$\Gamma = D_m$ ($m \ge 4$) and $G =\langle g \rangle$ where  
\[
g(s_i) = s_i\ (1 \le i \le m-2)\,,\ g(s_{m-1}) = s_m\,,\ g(s_m) = s_{m-1}\,.
\]
\item[(iii)]
$\Gamma = D_4$ and $\langle g_1 \rangle \subset G \subset \langle g_1, g_2 \rangle$ where 
\begin{gather*}
g_1 (s_1) = s_3\,, \ g_1 (s_2) = s_2\,,\ g_1 (s_3) = s_4\,,\ g_1 (s_4) = s_1\,,\\
g_2 (s_1) = s_1\,,\ g_2 (s_2) = s_2\,,\ g_2 (s_3) = s_4\,,\ g_2 (s_4) = s_3\,.
\end{gather*}
\item[(iv)]
$\Gamma = E_6$ and $G = \langle g \rangle$ where 
\[
g (s_1) = s_6\,,\ g (s_2) = s_2\,,\ g(s_3) = s_5\,,\ g (s_4) = s_4\,,\ g (s_5) = s_3\,,\ g(s_6) = s_1\,.
\]
\item[(v)]
$\Gamma = \tilde A_{2 m +1}$ ($m \ge 1$) and $G = \langle g \rangle$ where 
\[
g(s_0) = s_0\,,\ g(s_i) = s_{2 m +2 -i}\ (1 \le i \le 2 m + 1)\,.
\]
\item[(vi)]
$\Gamma = \tilde D_m$ ($m \ge 4$) and $G =\langle g \rangle$ where  
\[
g(s_i) = s_i\ (0 \le i \le m-2)\,,\ g(s_{m-1}) = s_m\,,\ g(s_m) = s_{m-1}\,.
\]
\item[(vii)]
$\Gamma = \tilde D_4$ and $\langle g_1 \rangle \subset G \subset \langle g_1, g_2 \rangle$ where 
\begin{gather*}
g_1(s_0) = s_0\,,\ g_1 (s_1) = s_3\,, \ g_1 (s_2) = s_2\,,\ g_1 (s_3) = s_4\,,\ g_1 (s_4) = s_1\,,\\
g_2(s_0) = s_0\,,\ g_2 (s_1) = s_1\,,\ g_2 (s_2) = s_2\,,\ g_2 (s_3) = s_4\,,\ g_2 (s_4) = s_3\,.
\end{gather*}
\item[(viii)]
$\Gamma = \tilde E_6$ and $G = \langle g \rangle$ where
\end{itemize}
\[
g(s_0) = s_0\,,\ g (s_1) = s_6\,,\ g (s_2) = s_2\,,\ g(s_3) = s_5\,,\ g (s_4) = s_4\,,\ g (s_5) = s_3\,,\ g(s_6) = s_1\,.
\]
\begin{itemize}
\item[(ix)]
$\Gamma = {}_\infty A_\infty$ and $G = \langle g \rangle$ where 
\[
g(s_i) = s_{-i}\ (i \in \Z)\,.
\]
\item[(x)]
$\Gamma = D_\infty$ and $G = \langle g \rangle$ where 
\[
g(s_1) = s_2\,,\ g(s_2) = s_1\,,\ g(s_i) = s_i\ (i \ge 3)\,.
\]
\end{itemize}
\end{thm}

\begin{figure}[ht!]
\begin{center}
\begin{tabular}{cc}
\parbox[c]{5cm}{\includegraphics[width=4.8cm]{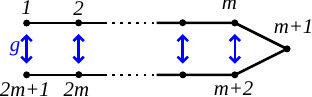}}&
\parbox[c]{4.8cm}{\includegraphics[width=4.4cm]{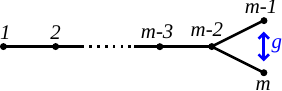}}\\
\noalign{\smallskip}
$(i)\quad (A_{2m+1}, \langle g \rangle)$&
$(ii)\quad (D_m, \langle g \rangle)$
\end{tabular}

\bigskip
\begin{tabular}{cc}
\parbox[c]{2.6cm}{\includegraphics[width=2.4cm]{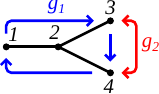}}&
\parbox[c]{3cm}{\includegraphics[width=2.8cm]{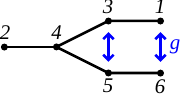}}\\
\noalign{\smallskip}
$(iii)\quad (D_4, \langle g_1, g_2 \rangle)$&
$(iv)\quad (E_6, \langle g \rangle)$
\end{tabular}

\caption{Pairs with the $\hat \Phi^+$-basis property: spherical type cases}\label{fig2_2}
\end{center}
\end{figure}

\begin{figure}[ht!]
\begin{center}
\begin{tabular}{cc}
\parbox[c]{5.4cm}{\includegraphics[width=5.2cm]{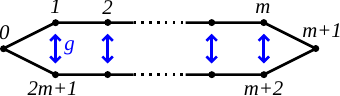}}&
\parbox[c]{5.2cm}{\includegraphics[width=5cm]{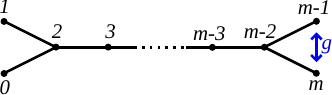}}\\
\noalign{\smallskip}
$(v)\quad (\tilde A_{2m+1}, \langle g \rangle)$&
$(vi)\quad (\tilde D_m, \langle g \rangle)$
\end{tabular}

\bigskip
\begin{tabular}{cc}
\parbox[c]{2.6cm}{\includegraphics[width=2.4cm]{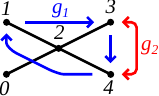}}&
\parbox[c]{3.8cm}{\includegraphics[width=3.6cm]{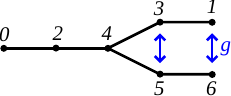}}\\
\noalign{\smallskip}
$(vii)\quad (\tilde D_4, \langle g_1, g_2 \rangle)$&
$(viii)\quad (\tilde E_6, \langle g \rangle)$
\end{tabular}

\caption{Pairs with the $\hat \Phi^+$-basis property: affine type cases}\label{fig2_3}
\end{center}
\end{figure}

\begin{figure}[ht!]
\begin{center}
\begin{tabular}{cc}
\parbox[c]{4.6cm}{\includegraphics[width=4.4cm]{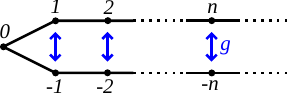}}&
\parbox[c]{5cm}{\includegraphics[width=4.8cm]{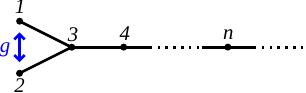}}\\
\noalign{\smallskip}
$(ix)\quad ({}_\infty A_\infty, \langle g \rangle)$&
$(x)\quad (D_\infty, \langle g \rangle)$
\end{tabular}

\caption{Pairs with the $\hat \Phi^+$-basis property: locally spherical type cases}\label{fig2_4}
\end{center}
\end{figure}

\subsection{Linear representations}

We return to our initial motivation before starting the proofs. 
Recall that a Coxeter graph $\Gamma$ is called \emph{simply laced} if $m_{s,t} \in \{2,3\}$ for all $s,t \in S$, $s \neq t$, and that $\Gamma$ is called \emph{triangle free} if there are no three distinct vertices $s,t,r \in S$ such that $m_{s,t}, m_{t,r}, m_{r,s} \ge 3$.
Suppose that $\Gamma$ is a finite, simply laced and triangle free Coxeter graph. 
Let $A_\Gamma$ be the Artin group and $A_\Gamma^+$ be the Artin monoid associated with $\Gamma$.
Suppose that $\K = \Q(q,z)$ and $E$ is a vector space over $\K$ having a basis $\BB = \{ e_\alpha \mid \alpha \in \Phi^+\}$ in one-to-one correspondence with the set $\Phi^+$ of positive roots of $\Gamma$.
By Krammer \cite{Kramm1}, Cohen--Wales \cite{CohWal1}, Digne \cite{Digne1} and Paris \cite{Paris1}, there is a linear representation $\psi : A_\Gamma \to \GL (E)$ which is faithful if $\Gamma$ is of spherical type and  which is always faithful on the monoid $A_\Gamma^+$.
Let $G$ be a group of symmetries of $\Gamma$.
Recall from Subsection \ref{subsec1_1} that $W_\Gamma^G$ is a Coxeter group associated with a precise Coxeter graph $\hat \Gamma$ and by Crisp \cite{Crisp1, Crisp2} we have $(A_\Gamma^+)^G = A_{\hat \Gamma}^+$.
Then $G$ acts on $E$, the linear representation $\psi$ is equivariant, and it induces a linear representation $\psi^G : A_{\hat \Gamma} \to \GL (E^G)$.
By Castella \cite{Caste2, Caste3} this representation is always faithful on $A_{\hat \Gamma}^+$ and is faithful on the whole $A_{\hat \Gamma}$ if $\Gamma$ is of spherical type.

One can find an explicit description of $\hat \Gamma$ in Crisp \cite{Crisp1, Crisp2} and in Geneste--Paris \cite{GenPar1}.
In particular, we have $\hat \Gamma = B_{m+1}$ in Case (i) of Theorem \ref{thm2_12}, we have $\hat \Gamma = B_{m-1}$ in Case (ii), $\hat \Gamma = G_2$ in Case (iii), $\hat \Gamma = F_4$ in Case (iv), $\hat \Gamma = \tilde C_{m+1}$ if $m \ge 2$ and $\hat \Gamma = \tilde B_2$ if $m=1$ in Case (v), $\hat \Gamma = \tilde B_{m-1}$ in Case (vi), $\hat \Gamma = \tilde G_2$ in Case (vii), and $\hat \Gamma = \tilde F_4$ in Case (viii).

So, concerning a description of a linear representation $\psi : A_\Gamma \to \GL(E)$ as above, where $E$ has a given basis $\BB = \{ e_\alpha \mid \alpha \in \Phi^+\}$ in one-to-one correspondence with the set $\Phi^+$ of positive roots of $\Gamma$, for $\Gamma$ of spherical or affine type, the situation is as follows.
For the following Coxeter graphs of spherical type the construction is done and the representation is faithful on the whole group $A_\Gamma$.
\begin{itemize}
\item
$A_m$, ($m \ge 1$): original work of Krammer \cite{Kramm1}.
\item
$D_m$ ($m \ge 4$), $E_6,\ E_7,\ E_8$: due to Digne \cite{Digne1} and independently to Cohen--Wales \cite{CohWal1}. 
\item
$B_m$ ($m \ge 2$), $F_4,\ G_2$: due to Digne \cite{Digne1}.
\end{itemize}
Such representations for the Coxeter graphs $H_3$, $H_4$ and $I_2(p)$ ($p=5$ or $p \ge 7$) are unknown.
For the following Coxeter graphs of affine type the construction is done and the representation is faithful on the Artin monoid $A_\Gamma^+$.
\begin{itemize}
\item
$\tilde A_m$ ($m \ge 3$), $\tilde D_m$ ($m \ge 4$), $\tilde E_6$, $\tilde E_7$, $\tilde E_8$: due to Paris \cite{Paris1}.
\item
$\tilde A_2$: due to Castella \cite{Caste1}.
\item
$\tilde B_m$ ($m \ge 2$), $\tilde C_m$ ($m \ge 3$), $\tilde F_4$, $\tilde G_2$: due Castella \cite{Caste2, Caste3} for the construction of the representation $\psi : A_\Gamma \to \GL (E)$ and for the proof of the faithfulness of $\psi$ on $A_\Gamma^+$, and due to the present work for an explicit construction of a basis in one-to-one correspondence with $\Phi^+$.
\end{itemize}
Curiously a construction for the remaining Coxeter graph of affine type, $\tilde A_1$, is unknown.

\subsection{Organization of the paper}

The paper is organized as follows.
In Subsection \ref{subsec2_1} and Subsection \ref{subsec2_2} we give preliminaries on root systems and symmetries.
In Section \ref{newsec3} we define the subset $\hat \BB = \{ \hat e_{\hat \beta} \mid \hat \beta \in \hat \Phi^+\}$ of $E^G$.
Section \ref{sec3} is dedicated to the proofs.
Theorem \ref{thm2_9} is proved in Subsection \ref{subsec3_1} and Theorem \ref{thm2_10} is proved in Subsection \ref{subsec3_2}.
The proof of Theorem \ref{thm2_12} is divided into two parts.
In a first part (see Subsection \ref{subsec3_3}) we show that, under the assumptions ``finite orbits and $\Gamma$ connected'', the $\hat \Phi^+$-basis property is quite restrictive.
More precisely we prove the following.

\begin{prop}\label{prop2_11}
Suppose that $\Gamma$ is a connected Coxeter graph, $G$ is a non-trivial group of symmetries of $\Gamma$, and all the orbits of $S$ under the action of $G$ are finite.
If $(\Gamma, G)$ has the $\hat \Phi^+$-basis property, then $\Gamma$ is one of the following Coxeter graphs: $A_{2m+1}$ ($m \ge 1$), $D_m$ ($m \ge 4$), $E_6$, $\tilde A_{2m+1}$ ($m \ge 1$), $\tilde D_m$ ($m \ge 4$), $\tilde E_6$, $\tilde E_7$, ${}_\infty A_\infty$, $D_\infty$.
\end{prop}

In a second part (see Subsection \ref{subsec3_4}) we study all the possible pairs $(\Gamma, G)$ with $\Gamma$ in the list of Proposition \ref{prop2_11} and $G$ non-trivial to prove Theorem \ref{thm2_12}.

\section{Preliminaries}

\subsection{Root systems}\label{subsec2_1}

Let $S$ be a (finite or infinite) set. 
A \emph{Coxeter matrix} on $S$ is a square matrix $M = (m_{s,t})_{s,t \in S}$ indexed by the elements of $S$ with coefficients in $\N \cup \{\infty \}$, such that $m_{s,s} = 1$ for all $s \in S$, and $m_{s,t} = m_{t,s} \ge 2$ for all $s,t \in S$, $s\neq t$.
This matrix is represented by its \emph{Coxeter graph}, $\Gamma$, defined as follows.
The set of vertices of $\Gamma$ is $S$, two vertices $s,t$ are connected by an edge if $m_{s,t} \ge 3$, and this edge is labelled with $m_{s,t}$ if $m_{s,t} \ge 4$.

The \emph{Coxeter group} associated with $\Gamma$ is the group $W = W_\Gamma$ defined by the presentation
\[
W = \langle S \mid s^2=1 \text{ for all } s \in S\,,\ (st)^{m_{s,t}}=1 \text{ for all } s,t \in S\,,\ s \neq t\text{ and } m_{s,t} \neq \infty \rangle\,.
\]
The pair $(W,S)$ is called the \emph{Coxeter system} associated with $\Gamma$.

There are several notions of ``root systems'' attached to all Coxeter groups.
The most commonly used is that defined by Deodhar \cite{Deodh1} and taken up by Humphreys \cite{Humph1}.
As Geneste--Paris \cite{GenPar1} pointed out, this definition is not suitable for studying symmetries of Coxeter graphs.
In that case it is better to take the more general definition given by Krammer \cite{Kramm2, Kramm3} and taken up by Davis \cite{Davis1}, or the even more general one given by H\'ee \cite{Hee2}.
We will use the latter in this paper.

We call \emph{root prebasis} a quadruple $\B = (S,V,\Pi,\RR)$, where $S$ is a (finite or infinite) set, $\Pi= \{\alpha_s  \mid s \in S \}$ is a set in one-to-one correspondence with $S$, $V= \oplus_{s \in S} \R \alpha_s$ is a real vector space with $\Pi$ as a basis, and $\RR = \{ \sigma_s \mid s \in S\}$ is a collection of linear reflections of $V$ such that $\sigma_s (\alpha_s) = -\alpha_s$ for all $s \in S$.
For all $s,t \in S$ we denote by $m_{s,t}$ the order of $\sigma_s \sigma_t$.
Then $M=(m_{s,t})_{s,t \in S}$ is a Coxeter matrix.
The Coxeter group of this matrix is denoted by  $W(\B)$.
We have a linear representation $f : W(\B) \to \GL (V)$ which sends $s$ to $\sigma_s$ for all $s \in S$.
Since we do not need to specify the map $f$ in general, for $w \in W$ and $x \in V$, the vector $f(w)(x)$ will be simply written $w(x)$.
We set $\Phi(\B) = \{ w(\alpha_s) \mid s \in S \text{ and } w \in W \}$ and we denote by $\Phi(\B)^+$ (resp. $\Phi(\B)^-$) the set of elements $\beta\in\Phi(\B)$ that are written $\beta = \sum_{s \in S} \lambda_s \alpha_s$ with $\lambda_s \ge 0$ (resp. $\lambda_s \le 0$) for all $s \in S$.
We say that $\B$ is a \emph{root basis} if we have the disjoint union $\Phi(\B) = \Phi (\B)^+ \sqcup \Phi (\B)^-$.
In that case $\Phi (\B)$ is called a \emph{root system} and the elements of $\Phi(\B)^+$ (resp. $\Phi (\B)^-$) are called \emph{positive roots} (resp. \emph{negative roots}).
Finally, we say that $\B$ is a \emph{reduced root basis} and that $\Phi (\B)$ is a \emph{reduced root system} if, in addition, we have $\R \alpha_s \cap \Phi (\B) = \{ \alpha_s, -\alpha_s\}$ for all $s \in S$.

\begin{expl}
Let $\Gamma$ be a Coxeter graph associated with a Coxeter matrix $M=(m_{s,t})_{s,t \in S}$.
As before, we set $\Pi=\{\alpha_s \mid s \in S\}$ and $V= \oplus_{s \in S} \R \alpha_s$.
We define a symmetric bilinear form $\langle .,. \rangle : V \times V \to \R$ by
\[
\langle \alpha_s, \alpha_t \rangle = \left\{ \begin{array}{ll}
-2 \cos(\pi/m_{s,t}) &\text{if } m_{s,t} \neq \infty\,,\\
-2 &\text{if } m_{s,t} = \infty\,.
\end{array} \right.
\]
For each $s \in S$ we define a reflection $\sigma_s: V \to V$ by $\sigma_s (x) = x - \langle \alpha_s, x \rangle \alpha_s$ and we set $\RR = \{\sigma_s \mid s \in S\}$.
By Deodhar \cite{Deodh1}, $\B = (S,V,\Pi, \RR)$ is a root basis and, by Bourbaki \cite{Bourb1}, $W_\Gamma = W(\B)$ and the representation $f: W_\Gamma \to \GL (V)$ is faithful.
This root basis is reduced since we have $\langle \beta, \beta \rangle = 2$ for all $\beta \in \Phi (\B)$.
It is called the \emph{canonical root basis} of $\Gamma$ and the bilinear form $\langle .,. \rangle$ is called the \emph{canonical bilinear form} of $\Gamma$.
\end{expl}

In this subsection we give the results on root systems that we will need, and we refer to H\'ee \cite{Hee2} for the proofs and more results.

\begin{thm}[H\'ee \cite{Hee2}]\label{thm2_1}
Let $\B = (S,V,\Pi, \RR)$ be a root basis.
Then the induced linear representation $f : W(\B) \to \GL (V)$ is faithful.
\end{thm}

Let $(W,S)$ be a Coxeter system.
The word length of an element $w \in W$ with respect to $S$ will be denoted by $\lg (w) = \lg_S(w)$.
The set of \emph{reflections} of $(W,S)$ is defined to be $R=\{ wsw^{-1} \mid w \in W \text{ and } s \in S\}$.
The following explains why the basis $\BB = \{ e_\beta \mid \beta \in \Phi^+ \}$ of our vector space $E$ will depend only on the Coxeter graph (or on the Coxeter system) and not on the root system.

\begin{thm}[H\'ee \cite{Hee2}]\label{thm2_2}
Let $\B = (S,V,\Pi, \RR)$ be a reduced root basis, let $W=W(\B)$, and let $\Phi = \Phi (\B)$.
Let $R$ be the set of reflections of $(W,S)$.
Let $\beta \in \Phi$.
Let $w \in W$ and $s \in S$ such that $w (\alpha_s) = \beta$.
We have $\beta \in \Phi^+$ if and only if $\lg (ws) > \lg (w)$.
In that case the element $\varpi (\beta) = wsw^{-1} \in R$ does not depend on the choice of $w$ and $s$. 
Moreover, the map $\varpi : \Phi^+ \to R$ defined in this way is a one-to-one correspondence.
\end{thm}

Let $\Gamma$ be a Coxeter graph and let $(W,S)$ be its associated Coxeter system.
For $X \subset S$ we denote by $\Gamma_X$ the full subgraph of $\Gamma$ spanned by $X$ and by $W_X$ the subgroup of $W$ generated by $X$.
By Bourbaki \cite{Bourb1}, $(W_X,X)$ is the Coxeter system of $\Gamma_X$.
If $\B = (S,V,\Pi, \RR)$ is a reduced root basis, then we denote by $V_X$ the vector subspace of $V$ spanned by $\Pi_X = \{ \alpha_s \mid s \in X \}$ and we set $\RR_X = \{ \sigma_s|_{V_X} \mid s \in X\}$.

\begin{prop}[H\'ee \cite{Hee2}]\label{prop2_3}
Let $\B = (S,V,\Pi, \RR)$ be a reduced root basis and let $X \subset S$.
Let $W=W(\B)$.
Then $\sigma_s (V_X) = V_X$ for all $s \in X$, the quadruple $\B_X = (X,V_X, \Pi_X, \RR_X)$ is a reduced root basis, $\Phi(\B_X) = \Phi(\B) \cap V_X$, and $W_X = W(\B_X)$.
\end{prop}

The following is proved in Bourbaki \cite{Bourb1} and is crucial in many works on Coxeter groups. 
It is important to us as well.

\begin{prop}[Bourbaki  \cite{Bourb1}]\label{prop2_4}
The following conditions on an element $w_0 \in W$ are equivalent.
\begin{itemize}
\item[(i)]
For all  $u \in W$ we have $\lg(w_0) = \lg(u) + \lg(u^{-1}w_0)$.
\item[(ii)]
For all $s \in S$ we have $\lg(sw_0)<\lg(w_0)$.
\end{itemize}
Moreover, such an element $w_0$ exists if and only $W$ is finite. 
If $w_0$ satisfies (i) and/or (ii) then $w_0$ is unique, $w_0$ is involutive (i.e. $w_0^2=1$), and $w_0 S w_0 =S$.
\end{prop}

When $W$ is finite the element $w_0$ in Proposition \ref{prop2_4} is called \emph{the longest element} of $W$.


\subsection{Symmetries}\label{subsec2_2}

Recall that a \emph{symmetry} of a Coxeter graph $\Gamma$ associated with a Coxeter matrix $M = (m_{s,t})_{s,t \in S}$ is a permutation $g$ of $S$ satisfying $m_{g(s),g(t)} = m_{s,t}$ for all $s,t \in S$.
Recall also that $\Sym (\Gamma)$ denotes the group of all symmetries of $\Gamma$.
Let $G$ be a subgroup of $\Sym(\Gamma)$.
We denote by $\OO$ the set of orbits of $S$ under the action of $G$.
Two subsets of $\OO$ will play a special role.
First, the set $\OO_\fin$ consisting of finite orbits.
Then, the subset $\SS \subset \OO_\fin$ formed by the orbits $X \in \OO_\fin$ such that $W_X$ is finite.
For $X \in \SS$ we denote by $u_X$ the longest element of $W_X$ (see Proposition \ref{prop2_4}) and, for $X,Y \in \SS$, we denote by $\hat m_{X,Y}$ the order of $u_X u_Y$.
Note that $\hat M = (\hat m_{X,Y})_{X,Y \in \SS}$ is a Coxeter matrix.
Its Coxeter graph is denoted by $\hat \Gamma = \hat \Gamma^G$.
Finally, we denote by $W^G$ the subgroup of $W$ of fixed elements under the action of $G$.

\begin{thm}[ H\'ee \cite{Hee2}, M\"uhlherr \cite{Muhlh1}]\label{thm2_5}
Let $\Gamma$ be a Coxeter graph, let $G$ be a group of symmetries of $\Gamma$, and let $(W,S)$ be the Coxeter system associated with $\Gamma$.
Then $W^G$ is generated by $\SS_W = \{u_X \mid X \in \SS \}$ and $(W^G, \SS_W)$ is a Coxeter system associated with $\hat \Gamma$.
\end{thm}

Take a root basis $\B = (S,V,\Pi, \RR)$ such that $W=W(\B)$.
The action of $G$ on $S$ induces an action of $G$ on $V$ defined by $g(\alpha_s) = \alpha_{g(s)}$ for all $s \in S$ and $g \in G$.
We say that $\B$ is \emph{symmetric} with respect to $G$ if $\sigma_{g(s)} = g \sigma_s g^{-1}$ for all $s \in S$ and $g \in G$.
Note that the canonical root basis is symmetric whatever is $G$.
Suppose that $\B$ is symmetric with respect to $G$.
Then the linear representation $f : W \to \GL (V)$ associated with $\B$ is equivariant in the sense that $f(g(w)) = g\, f(w)\, g^{-1}$ for all $w \in W$ and $g \in G$.
So, it induces a linear representation $f^G: W^G \to \GL (V^G)$, where $V^G = \{ x \in V \mid g(x) = x \text{ for all } g \in G \}$.

For each $X \in \SS$ we set $\hat \alpha_X = \sum_{s \in X} \alpha_s$ and we denote by $\hat V$ the vector subspace of $V$ spanned by $\hat \Pi = \{ \hat \alpha_X \mid X \in \SS \}$.
We have $\hat V \subset V^G$ but we have no equality in general.
However, we have $f^G(w)(\hat V) = \hat V$ for all $w \in W^G$.
We set $\hat \sigma_X = f^G(u_X)|_{\hat V}$ for all $X \in \SS$ and $\hat \RR = \{\hat \sigma_X \mid X \in \SS\}$.

\begin{thm}[Hée \cite{Hee2}]\label{thm2_6}
The quadruple $\hat \B = (\SS, \hat V, \hat \Pi, \hat \RR)$ is a root basis which is reduced if $\B$ is reduced.
Moreover, we have $W(\hat \B) = W^G$.
\end{thm}

The root basis $\hat \B$ of Theorem \ref{thm2_6} will be called the \emph{equivariant root basis} of $\B/G$.


\section{Definition of $\hat \BB$}\label{newsec3}

From now on $\Gamma$ denotes a given Coxeter graph, $M = (m_{s,t})_{s,t \in S}$ denotes its Coxeter matrix, and $(W,S)$ denotes its Coxeter system. 
We take a group $G$ of symmetries of $\Gamma$ and a reduced root basis $\B = (S, V, \Pi, \RR)$ such that $W = W(\B)$.
We assume that $\B$ is symmetric with respect to $G$ and we use again the notations of Subsection \ref{subsec2_2} ($\hat \Gamma$, $\hat M = (\hat m_{X,Y})_{X,Y \in \SS}$, $W^G$, $\hat \B = (\SS, \hat V, \hat \Pi, \hat \RR)$, and so on).
We set $\Phi = \Phi (\B)$ and $\hat \Phi = \Phi (\hat \B)$.

Let $\K$ be a field. 
We denote by $E$ the vector space over $\K$ having a basis $\BB = \{ e_\alpha \mid \alpha \in \Phi^+ \}$ in one-to-one correspondence with the set $\Phi^+$ of positive roots. 
The group $G$ acts on $E$ via its action on $\Phi^+$ and we denote by $E^G$ the vector subspace of $E$ of fixed vectors under this action.
We denote by $\Omega$ the set of orbits and by $\Omega_\fin$ the set of finite orbits of $\Phi^+$ under the action of $G$.
For each $\omega \in \Omega_\fin$ we set $\hat e_\omega = \sum_{\alpha \in \omega} e_\alpha$.
It is easily shown that $\BB_0 = \{ \hat e_\omega \mid \omega \in \Omega_\fin \}$ is a basis of $E^G$.

The definition of $\hat \BB$ requires the following two lemmas. 
The proof of the first one is left to the reader.

\begin{lem}\label{lem2_7}
Set $V^+ = \sum_{s \in S} \R_+ \alpha_s$.
Let $x,y \in V^+$ and $s \in S$.
If $x + y \in \R_+ \alpha_s$, then $x,y \in \R_+ \alpha_s$.
\end{lem}

For $X \in \SS$ we set $\omega_X = \{ \alpha_s \mid s \in X \}$.
Note that $\omega_X \in \Omega_\fin$ for all $X \in \SS$.
More generally, we have $w (\omega_X) \in \Omega_\fin$ for all $X \in \SS$ and all $w \in W^G$.

\begin{lem}\label{lem2_8}
Let $X, X' \in \SS$ and $w,w' \in W^G$.
If $w (\hat \alpha_X) = w' (\hat \alpha_{X'})$, then $w (\omega_X) = w' (\omega_{X'})$.
\end{lem}

\begin{proof}
Up to replacing the pair $(w, w')$ by $(w'^{-1} w,1)$ we can assume that $w'=1$.
Then we have $w (\hat \alpha_X) = \hat \alpha_{X'}$ and we must show that $w (\omega_X) = \omega_{X'}$.
For that it suffices to show that the intersection of the two orbits $w (\omega_X)$ and $\omega_{X'}$ is non-empty.
Either all the roots $w(\alpha_s)$, $s \in X$, lie in $\Phi^+$, or all of them lie in $\Phi^-$.
Moreover, their sum $w (\hat \alpha_X) = \hat \alpha_{X'}$ lies in $V_{X'}^+$.
So, they all lie in $\Phi_{X'}^+$.
Similarly, since $w^{-1} (\hat \alpha_{X'}) = \hat \alpha_X$, all the roots $w^{-1} (\alpha_t)$, $t \in X'$, lie in $\Phi_{X}^+$.
Let $s \in X$.
We have $w (\alpha_s) = \sum_{t \in X'} \lambda_t \alpha_t$ with $\lambda_t \ge 0$ for all $t \in X'$.
Hence, we have $\alpha_s = \sum_{t \in X'} \lambda_t\, w^{-1} (\alpha_t)$ and all the vectors $\lambda_t\, w^{-1} (\alpha_t)$, $t \in X'$, lie in $V^+$.
By Lemma \ref{lem2_7} it follows that all these vectors lie in $\R_+ \alpha_s$.
But the family $\{ w^{-1} (\alpha_t) \mid t \in X'\}$ is linearly independent, hence only one $\lambda_t$ is nonzero.
Thus, there exists $t \in X'$ such that $\lambda_t >0$ and $w (\alpha_s) = \lambda_t \alpha_t$.
Since the root basis $\B$ is reduced, we have $\lambda_t =1$, hence $w (\alpha_s) = \alpha_t \in w (\omega_X) \cap \omega_{X'}$, which completes the proof.
\end{proof}

Now we can define a map $F = F_{\B, G}: \hat \Phi^+ \to \Omega_\fin$ as follows.
Let $\hat \alpha \in \hat \Phi^+$.
Let $X \in \SS$ and $w \in W^G$ such that $\hat \alpha = w (\hat \alpha_X)$.
Then we set $F (\hat \alpha) = w (\omega_X)$.
By Lemma \ref{lem2_8} the definition of this map does not depend on the choices of $w$ and $X$.
Moreover, it is easily shown that it is injective. 
Now, we set 
\[
\hat \BB = \{\hat e_{F (\hat \alpha)} \mid \hat \alpha \in \hat \Phi^+ \}\,,
\]
and we say that the pair $(\Gamma, G)$ has the \emph{$\hat\Phi^+$-basis property} if $\hat \BB$ is a basis of $E^G$.
Note that $(\Gamma, G)$ has the $\hat\Phi^+$-basis property if and only if $\hat \BB = \BB_0$.
Equivalently, $(\Gamma, G)$ has the $\hat \Phi^+$-basis property if and only if $F$ is a bijection (or a surjection).


\section{Proofs}\label{sec3}

\subsection{Proof of Theorem \ref{thm2_9}}\label{subsec3_1}

In this subsection we denote by $S'$ the union of the finite orbits of $S$ under the action of $G$, and we set $\Gamma' = \Gamma_{S'}$, $V' = V_{S'}$, $\Pi' = \Pi_{S'}$ and $\RR' = \RR_{S'}$.
We consider the root basis $\B' = \B_{S'} = (S', V', \Pi', \RR')$ and its root system $\Phi' = \Phi_{S'} = \Phi (\B')$.
Each symmetry $g \in G$ induces a symmetry of $\Gamma'$.
We denote by $G'$ the subgroup of $\Sym (\Gamma')$ of all these symmetries.

\begin{lem}\label{lem3_1}
Let $\alpha \in \Phi$.
The following assertions are equivalent.
\begin{itemize}
\item[(i)]
The orbit $\omega (\alpha)$ of $\alpha$ under the action of $G$ is finite.
\item[(ii)]
The root $\alpha$ lies in $\Phi'$.
\end{itemize}
\end{lem}

\begin{proof}
Suppose that the orbit $\omega (\alpha)$ is finite.
The \emph{support} of a vector $x = \sum_{s \in S} \lambda_s \alpha_s \in V$ is defined to be $\Supp (x) = \{ s \in S \mid \lambda_s \neq 0 \}$.
The union $X_\alpha$ of the supports of the roots $\beta \in \omega (\alpha)$ is a finite set and is stable under the action of $G$, hence $X_\alpha$ is a union of finite orbits. 
This implies that $X_\alpha \subset S'$, hence $\alpha \in \Phi \cap V'$, and therefore, by Proposition \ref{prop2_3}, $\alpha \in \Phi'$.

Suppose $\alpha \in \Phi'$.
There exist $t_1, t_2, \dots, t_n, s \in S'$ such that $\alpha = (t_1 t_2 \cdots t_n) (\alpha_s)$.
For each $g \in G$ we have $g (\alpha) = (g(t_1)\, g(t_2) \cdots g(t_n)) (\alpha_{g(s)})$.
The respective orbits of $t_1, t_2, \dots, t_n$ and $s$ are finite, hence the orbit $\omega (\alpha) = \{ g (\alpha) \mid g \in G \}$ is finite.
\end{proof}

\begin{proof}[Proof of Theorem \ref{thm2_9}]
We denote by $\Omega'$ the set of orbits of $\Phi'^+$ under the action of $G'$ and by $\Omega'_\fin$ the subset of finite orbits.
On the other hand we denote by $\hat \B'$ the equivariant root basis of $\B' / G'$.
By Lemma \ref{lem3_1} each orbit of $\Phi'^+$ under the action of $G'$ is finite and each finite orbit in $\Phi^+$ is contained in $\Phi'^+$, hence $\Omega'_\fin = \Omega' = \Omega_\fin$.
Moreover, since each element $X \in \SS$ is contained in $S'$ and $W^G$ is generated by $\SS_W = \{ u_X \mid X \in \SS \}$ (see Theorem \ref{thm2_5}), we have $(W_{S'})^{G'} = W^G$ and $\hat \B' = \hat \B$.
So, we have $\Phi (\hat \B')^+ = \Phi (\hat \B)^+$ and $F_{\B, G} = F_{\B', G'} : \Phi (\hat \B)^+ \to \Omega_\fin$.
Since we know that $(\Gamma, G)$ has the $\hat \Phi^+$-basis property if and only if $F_{\B, G}$ is a bijection, $(\Gamma, G)$ has the $\hat \Phi^+$-basis property if and only if $(\Gamma', G')$ has the $\hat \Phi^+$-basis property. 
\end{proof}


\subsection{Proof of Theorem \ref{thm2_10}}\label{subsec3_2}

From now on we assume that all the orbits of $S$ under the action of $G$ are finite.
Then, by Lemma \ref{lem3_1}, the orbits of $\Phi$ under the action of $G$ are also finite.

\begin{lem}\label{lem3_2}
We assume that for each root $\alpha \in \Phi^+$ there exist $s \in S$ and $w \in W^G$ such that $\alpha = w (\alpha_s)$.
Let $s, t \in S$ and $g \in G$ such that $t = g(s) \neq s$.
Then $m_{s,t} = 2$.
\end{lem}

\begin{proof}
Suppose instead that $m_{s,t} \ge 3$.
Set $\alpha = s (\alpha_t)$.
We have $\alpha = \alpha_t + \lambda \alpha_s$ where $\lambda > 0$.
In particular $\alpha \in \Phi^+$.
So, there exist $w \in W^G$ and $r \in S$ such that $w(\alpha) = \alpha_r$.
Since $t = g(s)$, we have $\alpha_t = g(\alpha_s)$, hence $w (\alpha_t) = w (g (\alpha_s)) = g (w (\alpha_s))$.
This shows that either both roots $w (\alpha_s)$ and $w (\alpha_t)$ lie in $\Phi^+$, or they both lie in $\Phi^-$.
Moreover, $w (\alpha) = w (\alpha_t) + \lambda\, w(\alpha_s) = \alpha_r$, hence the two roots $w (\alpha_s)$ and $w (\alpha_t)$ lie in $\Phi^+$.
Thus, by Lemma \ref{lem2_7}, the two vectors $w (\alpha_t)$ and $\lambda\, w(\alpha_s)$ lie in $\R_+ \alpha_r$, which is a contradiction since, $t$ being different from $s$, these two vectors are linearly independent.
\end{proof}

\begin{prop}\label{prop3_3}
The following conditions are equivalent.
\begin{itemize}
\item[(i)]
The pair $(\Gamma, G)$ has the $\hat \Phi^+$-basis property. 
\item[(ii)]
For each root $\alpha \in \Phi^+$ there exist $w \in W^G$ and $s \in S$ such that $\alpha = w (\alpha_s)$.
\item[(iii)]
For each root $\alpha \in \Phi$ there exist $w \in W^G$ and $s \in S$ such that $\alpha = w (\alpha_s)$.
\end{itemize}
\end{prop}

\begin{proof}
Suppose that $(\Gamma, G)$ has the $\hat \Phi^+$-basis property.
Let $\alpha \in \Phi^+$.
The orbit $\omega (\alpha)$ lies in $\Omega_\fin = \Omega$ and the map $F : \hat\Phi^+ \to \Omega_\fin$ is a bijection, hence there exist $w \in W^G$ and $X \in \SS$ such that $w (\omega_X) = \omega (\alpha)$.
In particular, there exists an element $s \in X \subset S$ such that $w (\alpha_s) = \alpha$.

Suppose that for each $\alpha \in \Phi^+$ there exist $w \in W^G$ and $s \in S$ such that $\alpha = w (\alpha_s)$. 
Let $\omega \in \Omega_\fin = \Omega$.
Let $\alpha \in \omega$.
By assumption there exist $w \in W^G$ and $s \in S$ such that $\alpha = w (\alpha_s)$.
Let $X$ be the orbit of $s$ under the action of $G$.
The set $X$ is finite since it is an orbit and, by Lemma \ref{lem3_2}, $W_X$ is the direct product of $|X|$ copies of $\Z/2\Z$.
So, $W_X$ is finite and $X \in \SS$.
Set $\hat \beta = w (\hat \alpha_X)$.
We have $\hat \beta \in \hat \Phi^+$ and the orbit $F (\hat \beta) = w (\omega_X)$ contains the root $\alpha = w (\alpha_s)$, hence  it is equal to $\omega$.
So, the map $F$ is a surjection, hence $(\Gamma, G)$ has the $\hat \Phi^+$-basis property.

Suppose that for each $\alpha \in \Phi^+$ there exist $w \in W^G$ and $s \in S$ such that $\alpha = w (\alpha_s)$.
In order to show that for each $\alpha \in \Phi$ there exist $w \in W^G$ and $s \in S$ such that $\alpha = w (\alpha_s)$, it suffices to consider a root $\alpha \in \Phi^-$.
By assumption, since $- \alpha \in \Phi^+$, there exist $w' \in W^G$ and $s \in S$ such that $-\alpha = w' (\alpha_s)$.
Let $X$ be the orbit of $s$.
By Lemma \ref{lem3_2} the Coxeter graph $\Gamma_X$ is a finite union of isolated vertices, hence $W_X$ is finite, $u_X = \prod_{t \in X} t$, and $u_X (\alpha_t) = -\alpha_t$ for all $t \in X$.
Set $w = w' u_X$.
Then $w \in W^G$ and $w (\alpha_s) = w' u_X (\alpha_s) = w' (-\alpha_s) = \alpha$.

The implication (iii) $\Rightarrow$ (ii) is obvious.
\end{proof}

By combining Lemma \ref{lem3_2} and Proposition \ref{prop3_3} we get the following.

\begin{lem}\label{lem3_4}
Suppose that $(\Gamma, G)$ has the $\hat \Phi^+$-basis property. 
Let $X$ be an orbit of $S$ under the action of $G$.
Then $\Gamma_X$ is a finite union of isolated vertices.
In particular, $X \in \SS$, $u_X = \prod_{t \in X} t$, and $u_X (\alpha_t) = -\alpha_t$ for all $t \in X$.
\end{lem}

\begin{proof}[Proof of Theorem \ref{thm2_10}]
Let $\Gamma_i$, $i \in I$, be the connected components of $\Gamma$.
For $i \in I$ we denote by $S_i$ the set of vertices of $\Gamma_i$, and we set $V_i = V_{S_i}$, $\Pi_i = \Pi_{S_i}$, and $\RR_i = \RR_{S_i}$.
Consider the root basis $\B_i = \B_{S_i} = (S_i, V_i, \Pi_i, \RR_i)$ and its root system $\Phi_i = \Phi_{S_i} = \Phi (\B_i)$.
Note that $\Phi$ is the disjoint union of the $\Phi_i$'s and $W$ is the direct sum of the $W_{S_i}$'s.
Recall that each symmetry $g \in \Stab_G (\Gamma_i)$ induces a symmetry of $\Gamma_i$ and that $G_i$ denotes the subgroup of $\Sym (\Gamma_i)$ of these symmetries.

Suppose that $(\Gamma, G)$ has the $\hat \Phi^+$-basis property.
Let $i \in I$.
Let $\alpha \in \Phi_i$.
By Proposition \ref{prop3_3} there exist $s \in S$ and $w \in W^G$ such that $\alpha = w (\alpha_s)$.
Since $W$ is the direct sum of the $W_{S_j}$'s, $w$ is uniquely written $w = \prod_{j \in I} w_j$, where $w_j \in W_{S_j}$ for all $j \in I$, and there are only finitely many $j \in I$ such that $w_j \neq 1$.
Let $j \in I$ such that $s \in S_j$.
We have $\alpha = w (\alpha_s) = w_j (\alpha_s) \in \Phi_j$, hence $i = j$ and $s \in S_i$.
On the other hand, if $g \in \Stab_G(\Gamma_i)$, then $g(w) = w$ and $g(W_{S_i}) = W_{S_i}$, hence $g(w_i) = w_i$.
So, $w_i \in W_{S_i}^{G_i}$ and $\alpha = w_i (\alpha_s)$.
By Proposition \ref{prop3_3} we conclude that $(\Gamma_i, G_i)$ has the $\hat \Phi^+$-basis property.

Suppose that $(\Gamma_i, G_i)$ has the $\hat \Phi^+$-basis property for all $i \in I$.
Let $\alpha \in \Phi$.
Let $i \in I$ such that $\alpha \in \Phi_i$.
By Proposition \ref{prop3_3} there exist $w_i \in W_{S_i}^{G_i}$ and $s \in S_i$ such that $\alpha = w_i (\alpha_s)$.
The action of $G$ on $S$ induces an action of $G$ on $I$ defined by $g(\Gamma_i) = \Gamma_{g(i)}$, for $g \in G$ and $i \in I$.
Since the orbits of $S$ under the action of $G$ are finite, the orbits of $I$ under the action of $G$ are also finite.
Let $J \subset I$ be the orbit of $i$.
For each $j \in J$ we choose $g \in G$ such that $g(i) = j$ and we set $w_j = g(w_i) \in W_{S_j}$.
The fact that $w_i \in W_{S_i}^{G_i}$ implies that the definition of $w_j$ does not depend on the choice of $g$.
Let $w = \prod_{j \in J} w_j$.
Then $w \in W^G$ and $w (\alpha_s) = w_i (\alpha_s) = \alpha$.
By Proposition \ref{prop3_3} we conclude that $(\Gamma, G)$ has the $\hat \Phi^+$-basis property.
\end{proof}


\subsection{Proof of Proposition \ref{prop2_11}}\label{subsec3_3}

From now on we assume that $\Gamma$ is connected, that $G$ is nontrivial, and that all the orbits of $S$ under the action of $G$ are finite.
We also assume that $\Phi$ is the canonical root system and that $\langle .,. \rangle : V \times V \to \R$ is the canonical bilinear form of $\Gamma$.

Recall that for all $s,t \in S$ we have $\langle \alpha_s , \alpha_t \rangle = -2 \cos(\pi/m_{s,t})$.
In particular $\langle \alpha_s, \alpha_t \rangle = 2$ if $s=t$, $\langle \alpha_s, \alpha_t \rangle = 0$ if $m_{s,t}=2$, $\langle \alpha_s, \alpha_t \rangle = -1$ if $m_{s,t}=3$, and $-2 \le \langle \alpha_s, \alpha_t \rangle < -1$ if $4 \le m_{s,t} \le \infty$.
Let $\equiv$ denote the equivalence relation on $\Phi$ generated by the relation $\equiv_1$ defined by: $\alpha \equiv_1 \beta$ $\Leftrightarrow$ $\langle \alpha, \beta \rangle \not \in \{ 0, 1, -1 \}$.
Note that the relation $\equiv_1$ is reflexive (since $\langle \alpha, \alpha \rangle = 2$ for all $\alpha \in \Phi$) and symmetric.

The proof of Proposition \ref{prop2_11} is divided into two parts.
In a first part (see Corollary \ref{corl3_13}) we prove that, if $\equiv$ has at least two equivalence classes, then $\Gamma$ is one of the Coxeter graphs $A_m$ ($m \ge 1$), $D_m$ ($m \ge 4$), $E_m$ ($6 \le m \le 8$), $\tilde A_m$ ($m \ge 2$), $\tilde D_m$ ($m \ge 4$), $\tilde E_m$ ($6 \le m \le 8$), $A_\infty$, ${}_\infty A_\infty$, $D_\infty$.
In a second part (see Proposition \ref{prop3_17}) we prove that, if $(\Gamma, G)$ has the $\hat \Phi^+$-basis property, then $\equiv$ has at least two equivalence classes.

For $\alpha \in \Phi$ we denote by $r_\alpha : V \to V$ the linear reflection defined by $r_\alpha (x) = x - \langle \alpha , x \rangle \alpha$.
Note that, if $s \in S$ and $w \in W$ are such that $\alpha = w (\alpha_s)$, then $r_\alpha = w s w^{-1}$.
In particular $r_\alpha \in W$ for all $\alpha \in \Phi$ and $r_{\alpha_s} = s$ for all $s \in S$.

\begin{lem}\label{lem3_5}
\begin{itemize}
\item[(1)]
We have $\alpha \equiv_1 -\alpha$ for all $\alpha \in \Phi$.
\item[(2)]
Let $\alpha, \beta \in \Phi$ such that $\alpha \equiv \beta$.
Then $w (\alpha) \equiv w (\beta)$ for all $w \in W$.
\item[(3)]
Let $\alpha, \beta \in \Phi$ such that $\alpha \equiv \beta$.
Then $\alpha \equiv r_\alpha (\beta)$ and $\beta \equiv r_\alpha (\beta)$.
\end{itemize}
\end{lem}

\begin{proof}
Part (1) is true since $\langle \alpha, -\alpha \rangle = -2 \not \in \{ 0, 1, -1 \}$.
Part (2) follows from the fact that each $w \in W$ preserves the bilinear form $\langle .,. \rangle$.
Let $\alpha, \beta \in \Phi$ such that $\alpha \equiv \beta$.
By Part (2) we have $r_\alpha (\alpha) \equiv r_\alpha (\beta)$.
But $r_\alpha (\alpha) = -\alpha$ hence, by Part (1), $\alpha \equiv r_\alpha (\beta)$.
We also have $\beta \equiv r_\alpha (\beta)$ since $\alpha \equiv \beta$.
\end{proof}

\begin{lem}\label{lem3_6}
Assume that $\alpha_s \equiv \alpha_t$ for all $s,t \in S$.
Then all the elements of $\Phi$ are equivalent modulo the relation $\equiv$.
\end{lem}

\begin{proof}
It suffices to prove that for each $\alpha \in \Phi$ there exists $s \in S$ such that $\alpha \equiv \alpha_s$.
Let $\alpha \in \Phi$.
There exist $t \in S$ and $w \in W$ such that $\alpha = w (\alpha_t)$.
We argue by induction on the length of $w$.
We can assume that $w \neq 1$.
Then we have $w = sw'$ with $s \in S$, $w' \in W$ and $\lg (w') < \lg (w)$.
Set $\beta = w' (\alpha_t)$.
By induction there exists $u \in S$ such that $\beta \equiv \alpha_u$.
We have $\alpha = w (\alpha_t) = (sw') (\alpha_t) = s(\beta)$ hence, by Lemma \ref{lem3_5}\,(2), $\alpha \equiv s(\alpha_u)$.
But, by assumption, $\alpha_u \equiv \alpha_s$, hence, by Lemma \ref{lem3_5}\,(3), $\alpha_s \equiv s(\alpha_u) \equiv \alpha$.
\end{proof}

Recall that the \emph{support} of a vector $x = \sum_{s \in S} \lambda_s \alpha_s \in V$ is $\Supp (x) = \{ s \in S \mid \lambda_s \neq 0 \}$.
Let $s,t \in S$.
A \emph{path} from $s$ to $t$ of \emph{length} $\ell$ is a sequence $s_0, s_1, \dots, s_\ell$ in $S$ of length $\ell +1$ such that $s_0=s$, $s_\ell = t$ and $m_{s_{i-1}, s_i} \ge 3$ for all $i \in \{1, \dots, \ell\}$.
The \emph{distance} between $s$ and $t$, denoted by $d(s,t)$, is the shortest length of a path from $s$ to $t$.
Then the \emph{distance} between an element $s \in S$ and a subset $X \subset S$ is $d(s,X) = \min \{ d(s,t) \mid t \in X \}$.

\begin{lem}\label{lem3_7}
Let $\alpha = \sum_{s \in S} \lambda_s \alpha_s \in \Phi^+$, $t \in S \setminus \Supp (\alpha)$ and $t_0 \in \Supp (\alpha)$.
Assume that $d (t, \Supp (\alpha)) = d( t,t_0)$ and $\lambda_{t_0} >1$.
Then $\alpha \equiv \alpha_t$.
\end{lem}

\begin{proof}
We argue by induction on $d = d(t, \Supp (\alpha))$.
There exists a path $t_0, t_1, \dots, t_d$ of length $d$ from $t_0$ to $t$.
We first show that $\alpha \equiv_1 \alpha_{t_1}$.
We have $\langle \alpha, \alpha_{t_1} \rangle = \sum_{s \in S} \lambda_s \langle \alpha_s, \alpha_{t_1} \rangle$.
For each $s \in \Supp (\alpha)$ we have $\lambda_s >0$ and $\langle \alpha_s, \alpha_{t_1} \rangle \le 0$ (since $s \neq t_1$), hence $\langle \alpha, \alpha_{t_1} \rangle \le \lambda_{t_0} \langle \alpha_{t_0}, \alpha_{t_1} \rangle$.
Moreover, we have $\lambda_{t_0} > 1$ and $\langle \alpha_{t_0}, \alpha_{t_1} \rangle \le -1$ (since $m_{t_0, t_1} \ge 3$), hence $\langle \alpha, \alpha_{t_1} \rangle < -1$, and therefore $\alpha \equiv_1 \alpha_{t_1}$.
By Lemma \ref{lem3_5}\,(3) we also have $\alpha \equiv \alpha'$, where $\alpha' = t_1 (\alpha)$.

Now we can assume that $d \ge 2$.
We write $\alpha' = \sum_{s \in S} \lambda_s' \alpha_s$.
We have $\alpha' = \alpha - \langle \alpha, \alpha_{t_1} \rangle \alpha_{t_1}$, hence $\Supp (\alpha') = \Supp (\alpha) \cup \{t_1\}$, $\lambda_{t_1}' = - \langle \alpha, \alpha_{t_1} \rangle >1$, and $d (t, \Supp (\alpha')) = d (t, t_1) = d-1$, therefore, by induction, $\alpha' \equiv \alpha_t$.
Finally, we have $\alpha \equiv \alpha'$ and $\alpha' \equiv \alpha_t$, hence $\alpha \equiv \alpha_t$.
\end{proof}

\begin{lem}\label{lem3_8}
Assume that there are $s,t \in S$ such that $m_{s,t} \ge 4$.
Then the relation $\equiv$ has only one equivalence class.
\end{lem}

\begin{proof}
We have $\langle \alpha_s, \alpha_t \rangle = -2 \cos (\pi/ m_{s,t}) \le -2 \cos (\pi/4) = - \sqrt{2} < -1$, hence $\alpha_s \equiv \alpha_t$.
By Lemma \ref{lem3_5}\,(3) we also have $\alpha_s \equiv \alpha$ where $\alpha = s (\alpha_t)$.
Let $u \in S \setminus \{s, t\}$.
By Lemma \ref{lem3_6} it suffices to show that either $\alpha_u \equiv \alpha_s$ or $\alpha_u \equiv \alpha_t$.
We can and do assume that $d(u,s) \le d(u,t)$ and we show that $\alpha_u \equiv \alpha_s$.
Set $\alpha = \lambda_s \alpha_s + \alpha_t$.
We have $\Supp (\alpha) = \{s, t\}$, $d(u, \Supp (\alpha)) = d(u, s)$, and $\lambda_s = - \langle \alpha_s, \alpha_t \rangle > 1$, hence, by Lemma \ref{lem3_7}, $\alpha_u \equiv \alpha$.
So, since $\alpha_s \equiv \alpha$, we have $\alpha_u \equiv \alpha_s$.
\end{proof}

\begin{lem}\label{lem3_9}
If $S$ contains a subset $Y$ such that 
\begin{itemize}
\item[(a)]
$\emptyset \neq Y \neq S$, and
\item[(b)]
for all $s \in Y$ there exists $\alpha = \sum_{r \in Y} \lambda_r \alpha_r \in \Phi_Y^+$ such that $\alpha \equiv \alpha_s$ and $\lambda_r >1$ for all $r \in \Supp (\alpha)$,
\end{itemize}
then $\equiv$ has only one equivalence class.
\end{lem}

\begin{proof}
Let $s \in Y$ and $t \in S \setminus Y$.
There exists a root $\alpha = \sum_{r \in Y} \lambda_r \alpha_r \in \Phi_Y^+$ such that $\alpha \equiv \alpha_s$ and $\lambda_r >1$ for all $r \in \Supp (\alpha)$.
Then, by Lemma \ref{lem3_7}, $\alpha \equiv \alpha_t$.
So, since $\alpha \equiv \alpha_s$, we have $\alpha_s \equiv \alpha_t$.
Let $s, s' \in Y$.
Since $Y \neq S$, we can take $t \in S \setminus Y$.
By the above, $\alpha_s \equiv \alpha_t$ and $\alpha_{s'} \equiv \alpha_t$, hence $\alpha_s \equiv \alpha_{s'}$.
Let $t, t' \in S \setminus Y$.
Since $Y \neq \emptyset$, we can take $s \in Y$.
By the above, $\alpha_s \equiv \alpha_t$ and $\alpha_s \equiv \alpha_{t'}$, hence $\alpha_t \equiv \alpha_{t'}$.
We conclude by applying Lemma \ref{lem3_6}.
\end{proof}

\begin{lem}\label{lem3_10}
Suppose that $\Gamma$ is one of the following Coxeter graphs of affine type: $\tilde A_m$ ($m \ge 2$), $\tilde D_m$ ($ m \ge 4$), $\tilde E_6$, $\tilde E_7$, $\tilde E_8$.
For each $s \in S$ there exists a root $\alpha = \sum_{r \in S} \lambda_r \alpha_r \in \Phi^+$ such that $\langle \alpha_s , \alpha \rangle = -2$ and $\lambda_r \ge 2$ for all $r \in S$.
\end{lem}

\begin{proof}
We number the elements of $S$ as in Bourbaki \cite[Planches]{Bourb1} with the convention that the unnumbered vertex in Bourbaki \cite[Planches]{Bourb1} is here labelled with $0$.
Let $\beta$ denote the greatest root of $\Phi_{S \setminus \{ 0 \}}$.
Here is the value of $\beta$ according to $\Gamma$.
\begin{itemize}
\item
If $\Gamma = \tilde A_m$ ($ m \ge 2$), then $\beta = \alpha_1 + \alpha_2 + \cdots + \alpha_m$.
\item
If $\Gamma = \tilde D_m$ ($ m \ge 4$), then $\beta = \alpha_1 + 2 \alpha_2 + \cdots + 2 \alpha_{m-2} + \alpha_{m-1} + \alpha_m$.
\item
If $\Gamma = \tilde E_6$, then $\beta = \alpha_1 + 2 \alpha_2 + 2 \alpha_3 + 3 \alpha_4 + 2 \alpha_5 + \alpha_6$.
\item
If $\Gamma = \tilde E_7$, then $\beta = 2 \alpha_1 + 2 \alpha_2 + 3 \alpha_3 + 4 \alpha_4 + 3 \alpha_5 + 2 \alpha_6 + \alpha_7$.
\item
If $\Gamma = \tilde E_8$, then $\beta = 2 \alpha_1 + 3 \alpha_2 + 4 \alpha_3 + 6 \alpha_4 + 5 \alpha_5 + 4 \alpha_6 + 3 \alpha_7 + 2 \alpha_8$.
\end{itemize}
We observe in each case that $\langle \alpha_0, \beta \rangle = -2$.
In particular, we have $s_0 (\beta) = 2 \alpha_0 + \beta$.
Set $\delta = \alpha_0 + \beta$.
Then $\alpha_0 + \delta = s_0 (\beta) \in \Phi$.

Observe that all the coordinates of $\delta$ in the basis $\Pi$ are $\ge 1$.
On the other hand, for each root $\alpha \in \Phi$ we have $(\alpha + \delta) \in \Phi$ and $\langle \alpha, \delta \rangle = 0$.
Indeed, we see case by case that, for each $s \in S$, $\langle \alpha_s, \delta \rangle = 0$.
Subsequently $\langle \alpha, \delta \rangle = 0$ for all $\alpha \in \Phi$ and $w (\delta) = \delta$ for all $w \in W$.
Now, $\Gamma$ is connected and simply laced, hence for each $\alpha \in \Phi$ there exists $w \in W$ such that $\alpha = w (\alpha_0)$.
So, $\alpha + \delta = w (\alpha_0 + \delta) = (w s_0) (\beta) \in \Phi$.

Let $s \in S$.
Since $\Gamma$ is simply laced and connected there exists $w \in W$ such that $\alpha_s = w(\alpha_0)$.
We set $\alpha' = w(\beta)$ and we choose $k \ge 1$ so that all the coordinates of $\alpha = \alpha' + k \delta$ in the basis $\Pi$ are $\ge 2$.
Then, by the above, $\alpha \in \Phi$ and $\langle \alpha_s, \alpha \rangle = \langle \alpha_s, \alpha' \rangle = \langle w(\alpha_0), w (\beta) \rangle = \langle \alpha_0, \beta \rangle = -2$.
\end{proof}

\begin{corl}\label{corl3_11}
Suppose that there is a proper subset $Y$ of $S$ such that $\Gamma_Y$ is one of the following Coxeter graphs of affine type: $\tilde A_m$ ($m \ge 2$), $\tilde D_m$ ($m \ge 4$), $\tilde E_6$, $\tilde E_7$, $\tilde E_8$.
Then the relation $\equiv$ has only one equivalence class.
\end{corl}

\begin{proof}
This follows from Lemma \ref{lem3_9} and Lemma \ref{lem3_10}.
\end{proof}

\begin{lem}\label{lem3_12}
Suppose that $\Gamma$ is simply laced. 
Then one of the following three assertions is satisfied. 
\begin{itemize}
\item[(i)]
$\Gamma$ is one of the following Coxeter graphs of spherical type: $A_m$ ($m \ge 1$), $D_m$ ($m \ge 4$), $E_6$, $E_7$, $E_8$.
\item[(ii)]
$\Gamma$ is one of the following locally spherical Coxeter graphs: $A_\infty$, ${}_\infty A_\infty$, $D_\infty$.
\item[(iii)]
There exists a subset $Y$ of $S$ such that $\Gamma_Y$ is one of the following Coxeter graphs of affine type: $\tilde A_m$ ($m \ge 2$), $\tilde D_m$ ($m \ge 4$), $\tilde E_6$, $\tilde E_7$, $\tilde E_8$.
\end{itemize}
\end{lem}

\begin{proof}
If $\Gamma$ contains a circuit, then there is a subset $Y$ of $S$ such that $\Gamma_Y = \tilde A_m$ ($ m \ge 2$).
So, we can assume that $\Gamma$ is a tree.
If all the vertices of $\Gamma$ are of valence $\le 2$, then $\Gamma \in \{ A_m \mid m \ge 1 \} \cup \{A_\infty, {}_\infty A_\infty\}$.
So, we can assume that $\Gamma$ has at least one vertex of valence $\ge 3$.
If $\Gamma$ has at least two vertices of valence $\ge 3$, then there exists a subset $Y$ of $S$ such that $\Gamma_Y = \tilde D_m$ ($m \ge 5$).
So, we can assume that $\Gamma$ has a unique vertex $s_0$ of valence $\ge 3$.
If the valence of $s_0$ is $\ge 4$, then there is a subset $Y$ of $S$ such that $\Gamma_Y = \tilde D_4$.
So, we can assume that the valence of $s_0$ is $3$.
We denote by $\ell_1, \ell_2, \ell_3$ the (finite or infinite) lengths of the branches from $s_0$ and we assume that $1 \le \ell_1 \le \ell_2 \le \ell_3 \le \infty$.
If $\ell_1 \ge 2$, then there is a subset $Y$ of $S$ such that $\Gamma_Y = \tilde E_6$.
So, we can assume that $\ell_1 = 1$.
If $\ell_2 \ge 3$, then there is a subset $Y$ of $S$ such that $\Gamma_Y = \tilde E_7$.
If $\ell_2 = 1$ and $\ell_3 < \infty$, then $\Gamma = D_m$ where $m = \ell_3+3$.
If $\ell_2 = 1$ and $\ell_3 = \infty$, then $\Gamma = D_\infty$. 
So, we can assume that $\ell_2 = 2$.
If $\ell_3 = 2,3,4$, then $\Gamma = E_6, E_7, E_8$, respectively.
If $\ell_3 \ge 5$, then there is a subset $Y$ of $S$ such that $\Gamma_Y = \tilde E_8$.
\end{proof}

\begin{corl}\label{corl3_13}
If the relation $\equiv$ has at least two equivalence classes, then $\Gamma$ is one of the following Coxeter graphs: $A_m$ ($m \ge 1$), $D_m$ ($m \ge 4$), $E_m$ ($6 \le m \le 8$), $\tilde A_m$ ($ m \ge 2$), $\tilde D_m$ ($m \ge 4$), $\tilde E_m$ ($6 \le m \le 8$), $A_\infty$, ${}_\infty A_\infty$, $D_\infty$.
\end{corl}

\begin{proof}
This follows from Lemma \ref{lem3_8}, Corollary \ref{corl3_11} and Lemma \ref{lem3_12}.
\end{proof}

Now we start the second part of the proof of Proposition \ref{prop2_11}.

\begin{lem}\label{lem3_14}
Assume that $(\Gamma, G)$ has the $\hat \Phi^+$-basis property.
Let $\alpha, \beta \in \Phi$ and $g \in G$ such that $\beta = g (\alpha) \neq \alpha$.
Then $\langle \alpha, \beta \rangle = 0$.
\end{lem}

\begin{proof}
By Proposition \ref{prop3_3} there exist $w \in W^G$ and $s \in S$ such that $\alpha = w (\alpha_s)$.
Set $t = g(s)$.
We have $\alpha_t = g(\alpha_s)$, hence $\beta = g (\alpha) = g (w (\alpha_s)) = w (g (\alpha_s)) = w (\alpha_t)$.
Since $\beta \neq \alpha$, we have $\alpha_t \neq \alpha_s$, thus $t \neq s$.
By Lemma \ref{lem3_4} we have $m_{s,t} = 2$, hence $\langle \alpha_s, \alpha_t \rangle = 0$, and therefore $\langle \alpha, \beta \rangle = \langle w (\alpha_s), w (\alpha_t) \rangle = \langle \alpha_s, \alpha_t \rangle = 0$.
\end{proof}

\begin{lem}\label{lem3_15}
Suppose that $(\Gamma, G)$ has the $\hat \Phi^+$-basis property. 
Let $\alpha, \beta \in \Phi$ and $g \in G$ such that $g (\alpha) = \alpha$ and $\langle \alpha, \beta \rangle \not \in \{ 0, 1, -1 \}$.
Then $g (\beta) = \beta$.
\end{lem}

\begin{proof}
Set $\gamma = r_\beta (\alpha)$.
We have $\gamma = \alpha - \langle \alpha, \beta \rangle \beta$, hence $g (\gamma) = g (\alpha) - \langle \alpha, \beta \rangle g(\beta) = \alpha - \langle \alpha, \beta \rangle g(\beta)$, and therefore 
\begin{gather*}
\langle \gamma, g (\gamma) \rangle =
\langle \alpha - \langle \alpha, \beta \rangle \beta, \alpha - \langle \alpha, \beta \rangle g (\beta) \rangle =\\
\langle \alpha, \alpha \rangle - \langle \alpha, \beta \rangle \langle \beta, \alpha \rangle - \langle \alpha , \beta \rangle \langle \alpha, g (\beta) \rangle + \langle \alpha, \beta \rangle^2 \langle \beta, g (\beta) \rangle =\\
2 - 2 \langle \alpha, \beta \rangle^2 + \langle \alpha, \beta \rangle^2 \langle \beta, g (\beta) \rangle\,,
\end{gather*}
since $\langle \alpha, \alpha \rangle = 2$ and $\langle \alpha, g (\beta) \rangle = \langle g (\alpha), g (\beta) \rangle = \langle \alpha, \beta \rangle$.
Suppose that $g (\beta) \neq \beta$.
Then, since $\langle \alpha, \beta \rangle \neq 0$, we also have $g (\gamma) \neq \gamma$.
Then, by Lemma \ref{lem3_14}, we get $\langle \beta, g (\beta) \rangle = 0$ and $\langle \gamma, g(\gamma) \rangle = 0$.
It follows that $2 - 2 \langle \alpha, \beta \rangle^2 = 0$, hence $\langle \alpha, \beta \rangle^2 = 1$, which is a contradiction as $\langle \alpha, \beta \rangle \not\in \{ 1, -1 \}$.
So, $g (\beta) = \beta$.
\end{proof}

\begin{lem}\label{lem3_16}
Suppose that $(\Gamma, G)$ has the $\hat \Phi^+$-basis property. 
Let $g \in G$.
Then there exists $s \in S$ such that $g(s) = s$.
\end{lem}

\begin{proof}
Suppose instead that $g(s) \neq s$ for all $s \in S$.
Take an element $s$ in $S$ such that the distance $\ell = d (s, g(s))$ is minimal.
By Lemma \ref{lem3_4} we have $\ell \ge 2$.
Let $s = s_0, s_1, \dots, s_\ell = g(s)$ be a path of length $\ell$ from $s$ to $g(s)$.
The minimality of $\ell$ implies that the sets $\{ s_0, s_1, \dots, s_{\ell-1} \}$ and $\{ s_\ell = g(s_0), g(s_1), \dots, g(s_{\ell-1}) \}$ are disjoint. 
Set $\alpha = (s_0 s_1 \cdots s_{\ell-2}) (\alpha_{s_{\ell-1}})$ and $\beta = g(\alpha)$.
We have $\alpha = \lambda_0 \alpha_{s_0} + \lambda_1 \alpha_{s_1} + \cdots + \lambda_{\ell-2} \alpha_{s_{\ell-2}} + \alpha_{s_{\ell-1}}$ where $\lambda_i >0$ for all $i \in \{1, \dots, \ell-2\}$.
Thus $\beta = g(\alpha) = \lambda_0 \alpha_{g (s_0)} + \lambda_1 \alpha_{g (s_1)} + \cdots + \lambda_{\ell-2} \alpha_{g (s_{\ell-2})} + \alpha_{g (s_{\ell-1})}$.
In particular, $\beta \neq \alpha$. 
However, since $\Supp (\alpha) \cap \Supp (\beta) = \emptyset$, $s_{\ell-1} \in \Supp (\alpha)$, $s_\ell = g(s_0) \in \Supp (\beta)$, and $m_{s_{\ell-1}, s_\ell} \ge 3$, we have $\langle \alpha, \beta \rangle < 0$.
This contradicts Lemma \ref{lem3_14}.
\end{proof}

\begin{prop}\label{prop3_17}
Suppose that $(\Gamma, G)$ has the $\hat \Phi^+$-basis property.
Then the relation $\equiv$ has at least two equivalence classes.
\end{prop}

\begin{proof}
Suppose instead that $\equiv$ has a unique equivalence class.
Let $g \in G$.
Set $U = \{ x \in V \mid g(x)=x \}$.
By Lemma \ref{lem3_16} there exists $s \in S$ such that $g(s) = s$.
Then $g (\alpha_s) = \alpha_s$.
Let $\beta \in \Phi$.
By assumption we have $\alpha_s \equiv \beta$.
It follows from Lemma \ref{lem3_15} that $g (\beta) = \beta$.
Thus, $U$ contains $\Phi$, hence $U = V$.
So, $G = \{ \id \}$ which is a contradiction since we are under the assumption $G \neq \{ \id \}$. 
So, $\equiv$ has at least two equivalence classes.
\end{proof}

\begin{proof}[Proof of Proposition \ref{prop2_11}]
Suppose that $(\Gamma, G)$ has the $\hat \Phi^+$-basis property. 
By Proposition \ref{prop3_17} the relation $\equiv$ has at least two equivalence classes, hence, by Corollary \ref{corl3_13}, $\Gamma$ is one of the following Coxeter graphs: $A_m$ ($m \ge 1$), $D_m$ ($m \ge 4$), $E_m$ ($6 \le m \le 8$), $\tilde A_m$ ($m \ge 2$), $\tilde D_m$ ($m \ge 4$), $\tilde E_m$ ($6 \le m \le 8$), $A_\infty$, ${}_\infty A_\infty$, $D_\infty$.
The Coxeter graphs $A_1$, $E_7$, $E_8$, $\tilde E_8$ and $A_\infty$ have no nontrivial symmetry, hence $\Gamma$ is not one of these graphs.
The Coxeter graph $A_{2m}$ has no nontrivial symmetry fixing an element of $S$ and we know by Lemma \ref{lem3_16} that each element of $G$ fixes at least one element of $S$, hence neither $\Gamma$ is $A_{2m}$.
If $g$ is a nontrivial symmetry of the Coxeter graph $\tilde A_{2m}$ that fixes an element of $S$, then there exist $s,t \in S$ such that $m_{s,t}=3$ and $g(s) = t$, hence, by Lemma \ref{lem3_2}, such an element cannot lie in $G$.
Thus, $\Gamma$ is also different from $\tilde A_{2m}$.
So, $\Gamma$ is one of the following Coxeter graphs: $A_{2m+1}$ ($m \ge 1$), $D_m$ ($m \ge 4$), $E_6$, $\tilde A_{2m+1}$ ($m \ge 1$), $\tilde D_m$ ($m \ge 4$), $\tilde E_6$, $\tilde E_7$, ${}_\infty A_\infty$, $D_\infty$.
\end{proof}


\subsection{Proof of Theorem \ref{thm2_12}}\label{subsec3_4}

We assume again that $\Gamma$ is connected, that $G$ is non-trivial, and that the orbits of $S$ under the action of $G$ are finite.
We also assume that $\Phi$ is the canonical root system of $\Gamma$ and that $\langle .,. \rangle : V \times V \to \R$ is its canonical bilinear form.

\begin{lem}\label{lem3_18}
Assume that $(\Gamma, G)$ has the $\hat \Phi^+$-basis property.
Then $(\Gamma, G)$ is up to isomorphism one of the pairs given in the statement of Theorem \ref{thm2_12}.
\end{lem}

\begin{proof}
By Proposition \ref{prop2_11} we know that $\Gamma$ is one of the following Coxeter graphs: $A_{2m+1}$ ($m \ge 1$), $D_m$ ($m \ge 4$), $E_6$, $\tilde A_{2m+1}$ ($m \ge 1$), $\tilde D_m$ ($m \ge 4$), $\tilde E_6$, $\tilde E_7$, ${}_\infty A_\infty$, $D_\infty$.
By Lemma \ref{lem3_16} we also know that for each $g \in G$ there exists $s \in S$ such that $g(s) = s$.
If $\Gamma = \tilde A_{2m+1}$ and $g_1, g_2$ are two distinct symmetries of $\Gamma$ conjugated to the symmetry described in Case (v) of Theorem \ref{thm2_12}, then $g = g_1 g_2$ does not fix any element of $S$.
By Lemma \ref{lem3_16} such an element cannot lie in $G$.
Similarly, if $g_1, g_2$ are two distinct symmetries of ${}_\infty A_\infty$ conjugated to the symmetry described in Case (ix) of Theorem \ref{thm2_12}, then $g = g_1 g_2$ does not fix any element of $S$, hence such an element cannot lie in $G$.
Observe that, if $H$ is a subgroup of the symmetric group $\SSS_4$ which contains no cycle of length $4$ and no product of two disjoint transpositions, then $H$ fixes an element of $\{1,2,3,4\}$.
Using this observation it is easily seen that, when $\Gamma = \tilde D_4$, we are either in Case (vi) or Case (vii) of Theorem \ref{thm2_12}, or in the following Case (i) or Case (iii). 
So, it remains to prove that $(\Gamma, G)$ is not one of the following pairs (see Figure \ref{fig3_1}).
\begin{itemize}
\item[(i)]
$\Gamma = \tilde D_m$ ($m \ge 4$) and $G$ contains an element $g$ of order $2$ such that
\end{itemize}
\[
g(s_0) = s_1,\ g(s_1) = s_0,\ g(s_i) = s_i\ (2 \le i \le m-2),\ g(s_{m-1}) = s_m,\ g(s_m) = s_{m-1}\,.
\]
\begin{itemize}
\item[(ii)]
$\Gamma = \tilde D_{2m}$ ($m \ge 3$) and $G$ contains an element $g$ of order $2$ such that 
\[
g(s_i) = s_{2m-i}\ (0 \le i \le 2m)\,.
\]
\item[(iii)]
$\Gamma = \tilde D_4$ and $G$ contains an element $g$ of order $4$ such that 
\[
g(s_0) = s_4,\ g(s_1) = s_0,\ g(s_2)= s_2,\ g(s_3) = s_1,\ g(s_4) = s_3\,.
\]
\item[(iv)]
$\Gamma = \tilde E_6$ and $G$ contains an element $g$ of order $3$ such that
\end{itemize}
\[
g(s_0) = s_6,\ g(s_1) = s_0,\ g(s_2) = s_5,\ g(s_3) = s_2,\ g(s_4) = s_4,\ g(s_5) = s_3,\ g(s_6) = s_1\,.
\]
\begin{itemize}
\item[(v)]
$\Gamma = \tilde E_7$ and $G$ contains an element $g$ of order $2$ such that 
\begin{gather*}
g(s_0) = s_7,\ g(s_1) = s_6,\ g(s_2) = s_2,\ g(s_3) = s_5,\\
g(s_4) = s_4,\ g(s_5) = s_3,\ g(s_6) = s_1,\ g(s_7) = s_0\,.
\end{gather*}
\end{itemize}

\begin{figure}[ht!]
\begin{center}
\begin{tabular}{cc}
\parbox[c]{5.2cm}{\includegraphics[width=5cm]{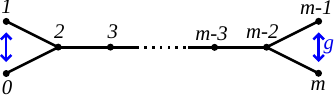}}&
\parbox[c]{5.6cm}{\includegraphics[width=5.4cm]{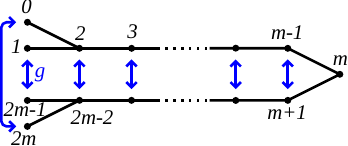}}\\
\noalign{\smallskip}
(i)&(ii)
\end{tabular}

\bigskip
\begin{tabular}{ccc}
\parbox[c]{2cm}{\includegraphics[width=1.8cm]{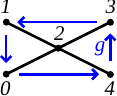}}&
\parbox[c]{3.6cm}{\includegraphics[width=3.4cm]{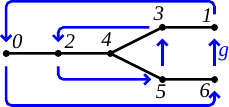}}&
\parbox[c]{3.5cm}{\includegraphics[width=3.3cm]{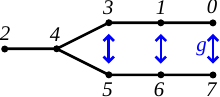}}\\
\noalign{\smallskip}
(iii)&(iv)&(v)
\end{tabular}

\caption{Pairs that do not have the $\hat \Phi^+$-basis property}\label{fig3_1}
\end{center}
\end{figure}

In each of the five cases we show a root $\alpha \in \Phi$ such that $g(\alpha) \neq \alpha$ and $\langle \alpha , g(\alpha) \rangle \neq 0$.
By Lemma \ref{lem3_14} this implies that $(\Gamma, G)$ does not have the $\hat \Phi^+$-basis property.
\begin{itemize}
\item[(i)]
We take $\alpha = \alpha_1 + \alpha_2 + \cdots + \alpha_{m-2} + \alpha_{m-1}$.
Then $g (\alpha) = \alpha_0 + \alpha_2 + \cdots + \alpha_{m-2} + \alpha_m \neq \alpha$ and $\langle \alpha, g(\alpha) \rangle = -2 \neq 0$.
\item[(ii)]
We take $\alpha = \alpha_0 + \alpha_1 + 2 \alpha_2 + \cdots + 2 \alpha_{m-1} + \alpha_m$.
Then $g (\alpha) = \alpha_m + 2 \alpha_{m+1} + \cdots + 2 \alpha_{2m -2} + \alpha_{2m-1} + \alpha_{2m} \neq \alpha$ and $\langle \alpha, g(\alpha) \rangle = -2 \neq 0$.
\item[(iii)]
We take $\alpha = \alpha_0 + \alpha_2 + \alpha_3$.
Then $g (\alpha) = \alpha_1 + \alpha_2 + \alpha_4 \neq \alpha$ and $\langle \alpha, g (\alpha) \rangle = -2 \neq 0$.
\item[(iv)]
We take $\alpha = \alpha_0 + \alpha_2 + \alpha_4 + \alpha_5$.
Then $g (\alpha) = \alpha_3 + \alpha_4 + \alpha_5 + \alpha_6 \neq \alpha$ and $\langle \alpha, g(\alpha) \rangle = -1 \neq 0$.
\item[(v)]
We take $\alpha = \alpha_1 + \alpha_2 + \alpha_3 + 2 \alpha_4 + 2 \alpha_5 + \alpha_6 + \alpha_7$.
Then $g (\alpha) = \alpha_0 + \alpha_1 + \alpha_2 + 2 \alpha_3 + 2 \alpha_4 + \alpha_5 + \alpha_6 + \alpha_7 \neq \alpha$ and $\langle \alpha, g (\alpha) \rangle = -2 \neq 0$.
\end{itemize}
\end{proof}

It remains to show that the pairs of Theorem \ref{thm2_12} have the $\hat \Phi^+$-basis property. 
We start with the pairs with a Coxeter graph of spherical type, that is, the first four cases. 
For $\alpha \in \Phi$ we denote by $W^G \alpha$ the orbit of $\alpha$ under the action of $W^G$.

\begin{lem}\label{lem3_19}
Suppose that $(\Gamma, G)$ is one of the pairs of Theorem \ref{thm2_12}.
Let $\alpha \in \Phi$.
\begin{itemize}
\item[(1)]
We have $W^G (-\alpha) = -(W^G\alpha)$.
\item[(2)]
For each $g \in G$ we have $g (W^G\alpha) = W^G g(\alpha)$.
\item[(3)]
For each $s \in S$ the orbit $W^G\alpha_s$ contains $-\alpha_s$ so that $W^G\alpha_s = W^G(-\alpha_s)$.
\end{itemize}
\end{lem}

\begin{proof}
Let $\alpha \in \Phi$ and $w \in W^G$.
Then $w (-\alpha) = - w(\alpha)$ and $g(w(\alpha)) = w (g (\alpha))$ for all $g \in G$.
This shows Part (1) and Part (2). 
Let $s \in S$ and let $X$ be the orbit of $s$ under the action of $G$.
Recall that $(\Gamma, G)$ is one of the pairs of Theorem \ref{thm2_12}.
Observe that, in each case, the Coxeter graph $\Gamma_X$ is a union of isolated vertices. 
Thus, $u_X = \prod_{t \in X} t$ and $u_X(\alpha_t) = -\alpha_t$ for all $t \in X$.
So, $u_X \in W ^G$ and $u_X (\alpha_s) = -\alpha_s$, hence $W^G \alpha_s = W^G (-\alpha_s)$.
\end{proof}

\begin{lem}\label{lem3_20}
If $(\Gamma, G)$ is one of the pairs given in Case (i), Case (ii), Case (iii) and Case (iv) of Theorem \ref{thm2_12}, then $(\Gamma, G)$ has the $\hat \Phi^+$-basis property.
\end{lem}

\begin{proof}
Note that $\Gamma$ is of spherical type in all the four cases. 
We number again the vertices of $\Gamma$ as in Bourbaki \cite[Planches]{Bourb1} (see also Figure \ref{fig2_2}) and we use the description of $\Phi^+$ given in that reference. 
By Proposition \ref{prop3_3} it suffices to show that there exists a subset $X \subset S$ such that $\Phi = \cup_{s \in X} W^G \alpha_s$.
We argue case by case.

{\it Case (i):}
We show that $\Phi = W^G \alpha_m \cup W^G \alpha_{m+1}$.
The orbits of $S$ under the action of $G$ are $X_1 = \{ 1, 2m+1 \}, X_2 = \{2, 2m \}, \dots, X_m = \{m, m+2\}$ and $\{m+1\}$.
Thus, by Theorem \ref{thm2_5}, $W^G = \langle u_{X_1}, \dots, u_{X_m}, s_{m+1} \rangle$.
By Lemma \ref{lem3_19} it suffices to show that $W^G \alpha_m \cup W^G \alpha_{m+1}$ contains $\Phi^+$.
First, we show that the orbit $W^G \alpha_{m+1}$ contains all the roots $\alpha \in \Phi^+$ such that $g(\alpha) = \alpha$.
Indeed, such a root either is equal to $\alpha_{m+1}$, or is of the form $\alpha = \sum_{i=k}^{2m+2-k} \alpha_i$ with $1 \le k \le m$.
In the second case we have $\alpha = (u_{X_k} \cdots u_{X_m})(\alpha_{m+1})$.
Now, we show that $W^G \alpha_m$ contains all the roots $\alpha \in \Phi^+$ such that $g(\alpha) \neq \alpha$, that is, the set of roots $\alpha = \sum_{i=1}^{2m+1} \lambda_i \alpha_i \in \Phi^+$ such that $\sum_{i=1}^m \lambda_i \neq \sum_{i=m+2}^{2m+1} \lambda_i$.
For $2 \le k \le m$ we have $\alpha_{k-1} = (u_{X_k} u_{X_{k-1}})(\alpha_k)$.
Hence $W^G \alpha_m$ contains the set $\{ \alpha_k \mid 1 \le k \le m \}$.
Now, let $1 \le k < \ell \le m$.
Then $\alpha_m + \alpha_{m+1} = s_{m+1} (\alpha_m) \in W^G \alpha_m$,
\[
\begin{array}{rcl}
\sum_{i=k}^\ell \alpha_i & = & (u_{X_\ell} \cdots u_{X_{k+1}}) (\alpha_k) \in W^G \alpha_m\,,\\
\noalign{\smallskip}
\sum_{i=k}^{m+1} \alpha_i & = & (s_{m+1} u_{X_m} \cdots u_{X_{k+1}}) (\alpha_k) \in W^G \alpha_m\,,\\
\noalign{\smallskip}
\sum_{i=k}^{2m+2-\ell} \alpha_i & = & (u_{X_\ell} \cdots u_{X_m} s_{m+1} u_{X_m} \cdots u_{X_{k+1}}) (\alpha_k) \in W^G \alpha_m\,.
\end{array}
\]
So, the orbit $W^G \alpha_m$ contains the set $\Psi$ of roots $\alpha = \sum_{i=1}^{2m+1} \lambda_i \alpha_i \in \Phi^+$ such that $\sum_{i=1}^m \lambda_i > \sum_{i=m+2}^{2m+1} \lambda_i$.
On the other hand, the orbit $W^G \alpha_m$ also contains the root $(s_{m+1} u_{X_m} s_{m+1})(\alpha_m) = \alpha_{m+2}$.
Thus, this orbit contains $\alpha_m$ and $g(\alpha_m) = \alpha_{m+2}$.
By Lemma \ref{lem3_19}\,(2), $W^G \alpha_m$ is stable by $g$, hence it contains the set $g (\Psi)$ of roots $\alpha = \sum_{i=1}^{2m+1} \lambda_i \alpha_i \in \Phi^+$ such that $\sum_{i=1}^m \lambda_i < \sum_{i=m+2}^{2m+1} \lambda_i$.
This ends the proof of Case (i).

{\it Case (ii):}
We show that $\Phi = W^G \alpha_{m-2} \cup W^G \alpha_m$.
The orbits of $S$ under the action of $G$ are $\{1\}, \{2\}, \dots, \{m-2\}$ and $X=\{m-1, m\}$.
Thus, by Theorem \ref{thm2_5},  $W^G = \langle s_1, \dots, s_{m-2}, u_X \rangle$.
By Lemma \ref{lem3_19} it suffices to show that $W^G \alpha_{m-2} \cup W^G \alpha_m$ contains $\Phi^+$.
The orbit $W^G \alpha_m$ contains $\alpha_m$.
For $1 \le k \le m-2$ it contains $(s_k \cdots s_{m-2}) (\alpha_m) = \sum_{i=k}^{m-2} \alpha_i + \alpha_m$.
It also contains $(s_{m-2} u_X s_{m-2})(\alpha_m) = \alpha_{m-1}$ and, for each $1 \le k \le m-2$, the root $(s_k \cdots s_{m-2}) (\alpha_{m-1}) = \sum_{i=k}^{m-2} \alpha_i + \alpha_{m-1}$.
So, the orbit $W^G \alpha_m$ contains all the roots $\alpha \in \Phi^+$ such that $g (\alpha) \neq \alpha$.
Now, we show that $W^G \alpha_{m-2}$ contains all the roots $\alpha \in \Phi^+$ such that $g(\alpha) = \alpha$.
For $1 \le k \le m-3$ we have $\alpha_k = (s_{k+1} s_k)(\alpha_{k+1})$.
Thus, the orbit $W^G \alpha_{m-2}$ contains the set $\{ \alpha_k \mid 1 \le k \le m-2 \}$.
It follows that it contains for each $1 \le \ell < k \le m-2$ the root $\sum_{i=\ell}^k \alpha_i = (s_\ell \cdots s_{k-1}) (\alpha_k)$.
On the other hand, $W^G \alpha_{m-2}$ contains $\alpha_{m-2} + \alpha_{m-1} + \alpha_m = u_X (\alpha_{m-2})$.
So, it contains $\sum_{i=k}^m \alpha_i = (s_k \cdots s_{m-3}) (\alpha_{m-2} + \alpha_{m-1} + \alpha_m)$ for each $1 \le k \le m-3$.
It also contains for each $1 \le k < \ell \le m-2$ the root 
\[
\sum_{i=k}^{\ell-1} \alpha_i + \sum_{i=\ell}^{m-2} 2 \alpha_i + \alpha_{m-1} + \alpha_m = (s_\ell \cdots s_{m-2}) \left( \sum_{i=k}^m \alpha_i \right)\,.
\]
This completes the proof of Case (ii).

{\it Case (iii):}
The orbits of $S$ under the action of $G$ are $X = \{1, 2, 3 \}$ and $\{ 2 \}$ hence, by Theorem \ref{thm2_5},
$W^G = \langle u_X, s_2 \rangle$.
We show that $\Phi = \cup_{i=1}^4 W^G \alpha_i$.
By Lemma \ref{lem3_19} it suffices to show that $\cup_{i=1}^4 W^G \alpha_i$ contains the $12$ positive roots of $\Phi$.
The orbit $W^G \alpha_1$ contains $\alpha_1$, $s_2 (\alpha_1) = \alpha_1 + \alpha_2$ and $(u_X s_2)(\alpha_1) = \alpha_2 + \alpha_3 + \alpha_4$.
Similarly, the orbit $W^G \alpha_3$ contains $\alpha_3$, $\alpha_2 + \alpha_3$ and $\alpha_1 + \alpha_2 + \alpha_4$, and the orbit $W^G \alpha_4$ contains $\alpha_4$, $\alpha_2 + \alpha_4$ and $\alpha_1 + \alpha_2 + \alpha_3$.
Finally, the orbit $W^G \alpha_2$ contains the roots $\alpha_2$, $u_X (\alpha_2) = \alpha_1 + \alpha_2 + \alpha_3 + \alpha_4$ and $(s_2 u_X) (\alpha_2) = \alpha_1 + 2 \alpha_2 + \alpha_3 + \alpha_4$.

{\it Case (iv):}
We show that $\Phi = W^G \alpha_3 \cup W^G \alpha_4$.
The orbits of $S$ under the action of $G$ are $X = \{1, 6\}, Y = \{3, 5\}, \{4\}, \{2\}$.
Thus, by Theorem \ref{thm2_5}, $W^G = \langle u_X, u_Y, s_4, s_2 \rangle$.
By Lemma \ref{lem3_19} it suffices to show that $W^G \alpha_3 \cup W^G \alpha_4$ contains the $36$ positive roots of $\Phi$.
The orbit $W^G \alpha_4$ contains the $12$ roots $\alpha \in \Phi^+$ such that $g(\alpha) = \alpha$, namely:
\begin{gather*}
\gamma_1 = \alpha_4,\
\gamma_2 = s_2 (\alpha_4) = \alpha_2 + \alpha_4,\
\gamma_3 = s_4 (\gamma_2) = \alpha_2,\
\gamma_4 = u_Y (\alpha_4) = \alpha_3 + \alpha_4 + \alpha_5,\\
\gamma_5 = s_2 (\gamma_4) = \alpha_2 + \alpha_3 + \alpha_4 + \alpha_5,\
\gamma_6 = u_X (\gamma_4) = \alpha_1  + \alpha_3 + \alpha_4 + \alpha_5 + \alpha_6,\\
\gamma_7 = s_2 (\gamma_6) = \alpha_1 + \alpha_2 + \alpha_3 + \alpha_4 + \alpha_5 + \alpha_6,\
\gamma_8 = s_4 (\gamma_5) = \alpha_2 + \alpha_3 + 2 \alpha_4 + \alpha_5,\\
\gamma_9 = s_4 (\gamma_7) = \alpha_1 + \alpha_2 + \alpha_3 + 2 \alpha_4 + \alpha_5 + \alpha_6,\\
\gamma_{10} = u_Y (\gamma_9) = \alpha_1 + \alpha_2 + 2 \alpha_3 + 2 \alpha_4 + 2 \alpha_5 + \alpha_6,\\
\gamma_{11} = s_4 (\gamma_{10}) = \alpha_1 + \alpha_2 + 2 \alpha_3 + 3 \alpha_4 + 2 \alpha_5 + \alpha_6,\\
\gamma_{12} = s_2 (\gamma_{11}) = \alpha_1 + 2 \alpha_2 + 2 \alpha_3 + 3 \alpha_4 + 2 \alpha_5 + \alpha_6\,.
\end{gather*}
Now, we show that the orbit $W^G \alpha_3$ contains the $24$ roots $\alpha \in \Phi^+$ such that $g (\alpha) \neq \alpha$.
First, it contains the following $12$ roots, that are the positive roots $\alpha = \sum_{i=1}^6 \lambda_i \alpha_i$ such that $\lambda_1 + \lambda_3 > \lambda_5 + \lambda_6$.
\begin{gather*}
\delta_1 = \alpha_3,\
\delta_2 = u_X (\delta_1) = \alpha_1 + \alpha_3,\
\delta_3 = u_Y (\delta_2) = \alpha_1,\
\delta_4 = s_4 (\delta_1) = \alpha_3 + \alpha_4,\\
\delta_5 = s_4 (\delta_2) = \alpha_1 + \alpha_3 + \alpha_4,\
\delta_6 = s_2 (\delta_4) = \alpha_2 + \alpha_3 + \alpha_4,\\
\delta_7 = s_2 (\delta_5) = \alpha_1 + \alpha_2 + \alpha_3 + \alpha_4,\
\delta_8 = u_Y (\delta_5) = \alpha_1 + \alpha_3 + \alpha_4 + \alpha_5,\\
\delta_9 = s_2 (\delta_8) = \alpha_1 + \alpha_2 + \alpha_3 + \alpha_4 + \alpha_5,\
\delta_{10} = s_4 (\delta_9) = \alpha_1 + \alpha_2 + \alpha_3 + 2 \alpha_4 + \alpha_5,\\
\delta_{11} = u_Y (\delta_{10}) = \alpha_1 + \alpha_2 + 2 \alpha_3 + 2 \alpha_4 + \alpha_5,\\
\delta_{12} = u_X (\delta_{11}) = \alpha_1 + \alpha_2 + 2 \alpha_3 + 2 \alpha_4 + \alpha_5 + \alpha_6\,.
\end{gather*}
On the other hand, the orbit $W^G \alpha_3$ contains the root $u_Y (\delta_4) = \alpha_4 + \alpha_5 = g (\delta_4)$.
So, this orbit contains $\delta_4$ and $g(\delta_4)$.
By Lemma \ref{lem3_19}\,(2) it is stable by $g$, hence it contains the images by $g$ of the $12$ above enumerated roots, that is, the positive roots $\alpha = \sum_{i=1}^6 \lambda_i \alpha_i$ such that $\lambda_1 + \lambda_3 < \lambda_5 + \lambda_6$.
This completes the proof of Case (iv).
\end{proof}

We turn now to study the pairs $(\Gamma, G)$ of Theorem \ref{thm2_12} where $\Gamma$ is of affine type, that is, the pairs of Case (v), Case (vi), Case (vii) and Case (viii).

\begin{lem}\label{lem3_21}
\begin{itemize}
\item[(1)]
Let $x,y \in V$ and $w \in \GL (V)$ such that $w (x) = x+y$ and $w(y) = y$.
Then $w^k (x) = x + ky$ for all $k \in \Z$.
\item[(2)]
Let $x,x',y \in V$ and $w,w' \in \GL(V)$ such that $w(x) = x' + y$, $w' (x') = x + y$ and $w (y) = w' (y) = y$.
Then $(w'w)^k (x) = x + 2k y$, $w (w'w)^k (x) = x' + (2k+1)y$, $(ww')^k (x') = x' + 2ky$, and $w' (ww')^k (x') = x + (2k+1)y$, for all $k \in \Z$.
\end{itemize}
\end{lem}

\begin{proof}
{\it Part (1):}
We prove the equality $w^k (x) = x + k y$ for $k \ge 0$ with an easy induction on $k$.
Indeed, $w^0(x) = x = x+ 0\,y$ and, if $w^k (x) = x + ky$, then $w^{k+1} (x) = w(x + ky) = x + (k+1)y$.
On the other hand, $x = w^{-1} (x+y) = w^{-1} (x) +y$, hence $w^{-1} (x) = x -y$.
By applying the above induction to $x$, $-y$ and $w^{-1}$, we see that the equality $w^k (x) = x + ky$ is also true for $k < 0$.

{\it Part (2):}
We have $(w' w) (x) = w' (x' + y) = x + 2y$.
Thus we obtain the first equality by applying Part (1) to $x$, $2y$ and $w'w$.
Now, $w (w'w)^k(x) = w (x + 2ky) = x' + (2k+1)y$.
We obtain the last two equalities by exchanging the roles of $x$ and $x'$ and those of $w$ and $w'$.
\end{proof}

\begin{lem}\label{lem3_22}
If $(\Gamma, G)$ is one of the pairs of Case (v), Case (vi), Case (vii) and Case (viii) of Theorem \ref{thm2_12}, then $(\Gamma, G)$ has the $\hat \Phi^+$-basis property.
\end{lem}

\begin{proof}
We set $S_1 = S \setminus \{ 0 \}$, $\Gamma_1 = \Gamma_{S_1}$, $\B_1 = \B_{S_1}$, $W_1 = W_{S_1} = W (\B_1)$ and $\Phi_1 = \Phi_{S_1} = \Phi (\B_1)$.
Note that the elements of $G$ fix $0$ and leave invariant $S_1$.
We denote by $G_1$ the subgroup of $\Sym (\Gamma_1)$ induced by $G$.

Let $\beta$ be the greatest root of $\Phi_1$ and let $\delta = \alpha_0 + \beta$.
Recall that $\langle \alpha, \delta \rangle = 0$ for all $\alpha \in \Phi$ and $w (\delta) = \delta$ for all $w \in W$ (see the proof of Lemma \ref{lem3_10}).
Here are the values of $\delta$ according to $\Gamma$.
\begin{itemize}
\item
If $\Gamma = \tilde A_{2m+1}$ ($m \ge 1$), then $\delta = \alpha_0 + \alpha_1 + \alpha_2 + \cdots + \alpha_{2m+1}$.
\item
If $\Gamma = \tilde D_m$ ($m \ge 4$), then $\delta = \alpha_0 + \alpha_1 + 2 \alpha_2 + \cdots + 2 \alpha_{m-2} + \alpha_{m-1} + \alpha_m$.
\item
If $\Gamma = \tilde E_6$, then $\delta = \alpha_0 + \alpha_1 + 2 \alpha_2 + 2 \alpha_3 + 3 \alpha_4 + 2 \alpha_5 + \alpha_6$.
\end{itemize}
The following claim is proved in Kac \cite[Proposition 6.3\,(a)]{Kac1}.

{\it Claim 1.}
{\it We have $\Phi = \{ \alpha + k \delta \mid \alpha \in \Phi_1 \text{ and } k \in \Z \}$.}

As in the proof of Lemma \ref{lem3_20} it suffices to show that there exists a subset $X \subset S$ such that $\Phi = \cup_{s \in X} W^G \alpha_s$.
For that we will use the following.

{\it Claim 2.}
{\it Let $Y$ be a subset of $S_1$ such that $\Phi_1 = \cup_{s \in Y} W_1^{G_1} \alpha_s$.
\begin{itemize}
\item[(1)]
Suppose that $\alpha_s + \delta \in W^G \alpha_s$ for all $s \in Y$.
Then $\Phi = \cup_{s \in Y} W^G \alpha_s$.
\item[(2)]
Suppose that there exists $t \in Y$ such that $\alpha_s + \delta \in W^G \alpha_s$ for all $s \in Y \setminus \{ t\}$, $\alpha_t + \delta \in W^G \alpha_0$ and $\alpha_0 + \delta \in W^G \alpha_t$.
Then $\Phi = \cup_{s \in Y \cup \{0\}} W^G \alpha_s$.
\end{itemize}}

{\it Proof of Claim 2.}
By Claim 1 each element of $\Phi$ is written $\alpha + k \delta$ with $\alpha \in \Phi_1$ and $k \in \Z$.
By assumption there exist $s \in Y$ and $w_1 \in W_1^{G_1}$ such that $\alpha = w_1 (\alpha_s)$.
Since $\delta$ is fixed by the elements of $W$ we have $\alpha + k \delta = w_1 (\alpha_s) + k \delta = w_1 (\alpha_s + k \delta)$.
Under the assumptions of Part (1) there exists $w \in W^G$ such that $\alpha_s + \delta = w(\alpha_s)$.
It follows from Lemma \ref{lem3_21}\,(1) that $\alpha_s + k \delta = w^k (\alpha_s)$, hence $\alpha + k \delta = w_1 (\alpha_s + k \delta) = (w_1 w^k) (\alpha_s) \in W^G \alpha_s$.
Now we assume the hypothesis of Part (2). 
If $s \neq t$, then we show as in the proof of Part (1) that $\alpha + k \delta \in W^G \alpha_s$.
Suppose that $s = t$.
Then there exist $w, w' \in W^G$ such that $\alpha_t + \delta = w (\alpha_0)$ and $\alpha_0 + \delta = w' (\alpha_t)$.
It follows from Lemma \ref{lem3_21}\,(2) that $\alpha_t + k \delta$ lies in $W^G \alpha_t$ if $k$ is even and $\alpha_t + k \delta$ lies in $W^G \alpha_0$ if $k$ is odd. 
Now, $\alpha + k \delta = w_1 (\alpha_t + k \delta)$, hence $\alpha + k \delta \in W^G \alpha_t \cup W^G \alpha_0$.
This ends the proof of Claim 2.

The rest of the proof of Lemma \ref{lem3_22} is case by case.

{\it Case (v):}
We show that $\Phi = W^G \alpha_0 \cup W^G \alpha_m \cup W^G \alpha_{m+1}$.
The orbits of $S$ under the action of $G$ are $\{0\}, X_1 = \{1, 2m+1\}, \dots, X_m = \{m, m+2\}, \{m+1\}$.
Thus $W^G = \langle s_0, u_{X_1}, \dots, u_{X_m}, s_{m+1} \rangle$.
We also know that $\Phi_1 = W_1^{G_1} \alpha_m \cup W_1^{G_1} \alpha_{m+1}$ (see the proof of Lemma \ref{lem3_20}).
So, by Claim 2 it suffices to show that $\alpha_0 + \delta \in W^G \alpha_{m+1}$, $\alpha_{m+1} + \delta \in W^G \alpha_0$ and $\alpha_m +  \delta \in W^G \alpha_m$.
This follows from the following formulas. 
\begin{gather*}
\alpha_0 + \delta = (s_0 u_{X_1} \cdots u_{X_m}) (\alpha_{m+1}),\
\alpha_{m+1} + \delta = (s_{m+1} u_{X_m} \cdots u_{X_1}) (\alpha_0),\\
\alpha_m + \delta = (u_{X_m} \cdots u_{X_1} s_0 u_{X_1} \cdots u_{X_m} s_{m+1}) (\alpha_m)\,.
\end{gather*}

{\it Case (vi):}
We show that $\Phi = W^G \alpha_{m-2} \cup W^G \alpha_m$.
The orbits of $S$ under the action of $G$ are $\{0\}, \{1\}, \{2\}, \dots, \{m-2\}, X = \{m-1, m\}$.
Thus $W^G = \langle s_0, s_1, s_2, \dots, s_{m-2}, u_X \rangle$.
On the other hand, we know that $\Phi_1 = W_1^{G_1} \alpha_{m-2} \cup W_1^{G_1} \alpha_m$ (see the proof of Lemma \ref{lem3_20}).
So, by Claim 2 it suffices to show that $\alpha_{m-2} + \delta \in W^G \alpha_{m-2}$ and $\alpha_m + \delta \in W^G \alpha_m$.
This follows from the following formulas.
\begin{gather*}
\alpha_{m-2} + \delta = (s_{m-2} s_{m-3} \cdots s_2 s_0 u_X s_1 s_2 \cdots s_{m-3}) (\alpha_{m-2}),\\
\alpha_m + \delta = (u_X s_{m-2} \cdots s_2 s_0 u_X s_1 s_2 \cdots s_{m-2}) (\alpha_m)\,.
\end{gather*}

{\it Case (vii):}
The orbits of $S$ under the action of $G$ are $X = \{1,3,4 \}, \{2\}, \{ 0 \}$, hence $W^G = \langle u_X, s_2, s_0 \rangle$.
We show that $\Phi = \cup_{i=1}^4 W^G \alpha_i$.
We know that $\Phi_1 = \cup_{i=1}^4 W_1^{G_1} \alpha_i$ (see the proof of Lemma \ref{lem3_20}).
So, by Claim 2 it suffices to show that $\alpha_i + \delta \in W^G \alpha_i$ for all $i \in \{ 1,2,3,4 \}$.
We have $\alpha_1 + \delta = (u_X s_2 s_0 u_X s_2)(\alpha_1) \in W^G \alpha_1$.
Similarly, $\alpha_3 + \delta \in W^G \alpha_3$ and $\alpha_4 + \delta \in W^G \alpha_4$.
Finally, $\alpha_2 + \delta = (s_2 s_0 u_X) (\alpha_2) \in W^G \alpha_2$.

{\it Case (viii):}
We show that $\Phi = W^G \alpha_3 \cup W^G \alpha_4$.
The orbits of $S$ under the action of $G$ are $X = \{ 1, 6\}, Y = \{ 3, 5 \}, \{0\}, \{2\}, \{4\}$.
Thus $W^G = \langle u_X, u_Y, s_0, s_2, s_4 \rangle$.
On the other hand $\Phi_1 = W_1^{G_1} \alpha_3 \cup W_1^{G_1} \alpha_4$ (see the proof of Lemma \ref{lem3_20}).
So, by Claim 2 it suffices to show that $\alpha_3 + \delta \in W^G \alpha_3$ and $\alpha_4 + \delta \in W^G \alpha_4$.
We use the notations of the proof of Case (iv) of Lemma \ref{lem3_20}.
We check that $\alpha_3 + \delta = (u_Y s_4 s_2 s_0)(g (\delta_{12}))$.
Thus, since $g (\delta_{12}) \in W_1^{G_1} \alpha_3$, we have $\alpha_3 + \delta \in W^G \alpha_3$.
Similarly, we check that $\alpha_4 + \delta = (s_4 s_2 s_0) (\gamma_{10})$.
Thus, since $\gamma_{10} \in W_1^{G_1} \alpha_4$, we have $\alpha_4 + \delta \in W^G \alpha_4$.
\end{proof}

The last two cases of Theorem \ref{thm2_12}, those with locally spherical Coxeter graphs (Case (ix) and Case (x)), are easily deduced from Case (i) and Case (ii) as we will see next.

\begin{lem}\label{lem3_23}
If $(\Gamma, G)$ is one of the pairs of Case (ix) and Case (x) of Theorem \ref{thm2_12}, then $(\Gamma, G)$ has the $\hat \Phi^+$-basis property.
\end{lem}

\begin{proof}
As ever, it suffices to show that for each $\alpha \in \Phi$ there exist $s \in S$ and $w \in W^G$ such that $\alpha = w (\alpha_s)$.

{\it Case (ix):}
Let $n \ge 1$.
We set $S_n = \{ -n, \dots, -1, 0 , 1, \dots, n \}$, $\Gamma_n = \Gamma_{S_n}$, $\B_n = \B_{S_n}$, $W_n = W_{S_n} = W (\B_n)$, $\Phi_n = \Phi_{S_n} = \Phi (\B_n)$ and $\Pi_n = \{ \alpha_s \mid s \in S_n\}$.
We denote by $V_n$ the vector subspace of $V$ spanned by $\Pi_n$.
Recall that $\Phi_n = V_n \cap \Phi$ (see Proposition \ref{prop2_3}). 
Note that the elements of $G$ leave invariant the set $S_n$.
We denote by $G_n$ the subgroup of $\Sym (\Gamma_n)$ induced by $G$.
Let $\alpha \in \Phi$.
Since $\alpha$ has finite support, there exists $n \ge 1$ such that $\alpha \in V_n \cap \Phi = \Phi_n$.
Then by Lemma \ref{lem3_20} there exist $s \in S_n \subset S$ and $w \in W_n^{G_n} \subset W^G$ such that $\alpha = w (\alpha_s)$.

{\it Case (x):}
We argue as in Case (ix).
Let $n \ge 4$.
We set $S_n = \{1,2, \dots, n \}$, $\Gamma_n = \Gamma_{S_n}$, $\B_n = \B_{S_n}$, $W_n = W_{S_n} = W (\B_n)$, $\Phi_n = \Phi_{S_n} = \Phi (\B_n)$ and $\Pi_n = \{ \alpha_s \mid s \in S_n\}$, and we denote by $V_n$ the vector subspace of $V$ spanned by $\Pi_n$.
The elements of $G$ leave invariant the set $S_n$, and we denote by $G_n$ the subgroup of $\Sym (\Gamma_n)$ induced by $G$.
Let $\alpha \in \Phi$.
Since $\alpha$ has finite support, there exists $n \ge 4$ such that $\alpha \in V_n \cap \Phi = \Phi_n$.
Then by Lemma \ref{lem3_20} there exist $s \in S_n \subset S$ and $w \in W_n^{G_n} \subset W^G$ such that $\alpha = w (\alpha_s)$.
\end{proof}

\begin{proof}[ Proof of Theorem \ref{thm2_12}]
It directly follows from Lemma \ref{lem3_18}, Lemma \ref{lem3_20}, Lemma \ref{lem3_22} and Lemma \ref{lem3_23}.
\end{proof}



\end{document}